\documentclass[10pt]{amsart} 

\usepackage[psamsfonts]{amssymb} 
\usepackage{amsfonts,amsmath} 
\usepackage{graphicx, color}
\usepackage{enumerate}
\usepackage{url}
\usepackage{hyperref}
\usepackage{tikz-cd}
\usepackage{ytableau}

\usepackage[all]{xy}

\newtheorem{theorem}{Theorem}[section]

\newtheorem{proposition}[theorem]{Proposition}
\newtheorem{lemma}[theorem]{Lemma}
\newtheorem{corollary}[theorem] {Corollary}

\newtheorem{question}[theorem]{Question}
\newtheorem*{claim*}{Claim}

\theoremstyle{remark}
\newtheorem{remark}[theorem]{Remark}

\theoremstyle{definition}
\newtheorem{definition}[theorem]{Definition}

\def\Z{\mathbb Z}
\def\R{\mathbb R}

\def\Q{\mathbb Q}
\def\C{\mathbb C}

\def\hom{{\mathrm{Hom}}}

\def\co{\colon}

\newcommand\abs[1]{\lvert#1\rvert}

\newcommand\sg[1]{\mathfrak{S}_{#1}}

\title[Quotients of $B_n$ that are extensions of $\sg{n}$]{Quotients of the braid group that are extensions of the symmetric group}
\author{Matthew B. Day and Trevor Nakamura}   

\date{\today}
\begin{document}
	
	\maketitle

\begin{abstract} 
We consider normal subgroups $N$ of the braid group $B_n$ such that the quotient $B_n/N$ is an extension of the symmetric group by an abelian group.
We show that, if $n\geq 4$, then there are exactly $8$ commensurability classes of such subgroups.
We define a \emph{Specht subgroup} to be a subgroup of this form that is maximal in its commensurability class.
We give descriptions of the Specht subgroups in terms of winding numbers and in terms of infinite generating sets.
The quotient of the pure braid group by a Specht subgroup is a module over the symmetric group.
We show that the modules arising this way are closely related to Specht modules for the partitions $(n-1,1)$ and $(n-2,2)$, working over the integers.
We compute the second cohomology of the symmetric group with coefficients in both of these Specht modules, working over an arbitrary commutative ring.
Finally, we determine which of the extensions of the symmetric group arising from Specht subgroups are split extensions.
\end{abstract}

\section{Introduction}
\subsection{Background}
We consider group extensions of the form
\begin{equation}\label{eq:topic} 0\to A\to B_n/N\to \sg{n}\to 1,\end{equation}
where $B_n$ is the braid group on $n\geq 3$ strands, $A$ an abelian group, $\sg{n}$ is the symmetric group, and $N\lhd B_n$ is a normal subgroup such that $B_n/N\to \sg{n}$ factors through the canonical projection $B_n\to \sg{n}$ (see section~\ref{se:bgprelims} for careful definitions).
In this setup, it follows that we have containments between $N$ and the pure braid group and its commutator subgroup: $PB_n'\leq N\leq PB_n$.
The group $A=PB_n/N$ gets the structure of a $\Z \sg{n}$-module from the conjugation action of $B_n$ on $PB_n$.
This paper is concerned with the range of possibilities for such extensions and the structure and cohomology of such modules.
In particular, we are interested in the conditions under which extension~\eqref{eq:topic} splits.

Extensions of the form~\eqref{eq:topic} arise naturally in more than one way.
The canonical projection $B_n\to \sg{n}$ is a natural thing to consider, since it is the map that takes braid diagrams and forgets the crossing data to obtain arrow diagrams for bijections.
It is canonical in a stronger sense, in that for $n\geq 3$ with $n\notin\{4,6\}$ it is the unique homomorphism $B_n\to\sg{n}$ with non-cyclic image, up to conjugacy.
This is a theorem of Artin~\cite{ArtinAnnals} that was recently reproven and strengthened by Kolay~\cite{Kolay}.
This also highlights the importance of the kernel of this map, which is the pure braid group  $PB_n$.
The pure braid group is the mysterious part of the braid group, and its abelianization is an obvious thing to consider in the study of $B_n$.
So one important case of extension~\eqref{eq:topic} is where $N=PB_n'$ and $A$ is the abelianization $H_1(PB_n)$ of $PB_n$.
In fact, this extension is initial in the category of extensions of this form.
As a $\Z\sg{n}$-module, $H_1(PB_n)$ is the permutation module $M^2_\Z$ of tabloids for the partition $(n-2,2)$, working over $\Z$.
This is an important connection between the group theory of $B_n$ and the representation theory of $\sg{n}$.

In~\cite{BrendleMargalit}, Brendle and Margalit study the level--$4$ subgroup $B_n[4]$ of the braid group.
This is the kernel of the modulo--$4$ reduction of the integral Burau representation.
Since $B_n$ is the mapping class group of a disk with $n$ marked points, and the integral Burau representation is the action of $B_n$ on the homology of a surface that is a particular branched cover of the disk, $B_n[4]$ arises from the topology of surfaces.
Further, Brendle and Margalit show that $B_n[4]$ is generated by the squares of pure braids, so that $B_n[4]$ comes up naturally in the study of $B_n$ without even referring to the Burau representation.
They also show, among other things, that $PB_n/B_n[4]$ is abelian, so that $B_n/B_n[4]$ is an extension of the form~\eqref{eq:topic}.

In~\cite{KordekMargalit}, Kordek and Margalit study $B_n[4]$ and $B_n/B_n[4]$.
Their main result is a stable description of $H^1(B_n[4];\C)$ as a sequence of $B_n/B_n[4]$-representations.
On the way, they prove the following result, which is a starting point for the ideas in this paper.
\begin{theorem}[Kordek--Margalit~\cite{KordekMargalit}, Proposition~7.6]
The short exact sequence
\[
0\to PB_n/B_n[4]\to B_n/B_n[4]\to \sg{n}\to 1
\]
does not split.
\end{theorem}

It is a classical fact that group extensions with cokernel $G$ and abelian kernel $A$ correspond to elements of the second cohomology $H^2(G;A)$ (see Theorem~\ref{th:EM} below).
Given such an extension, we call the corresponding element of $H^2(G;A)$ its \emph{structure class}.
%In order to understand $B_n/B_n[4]$ and $B_n/PB_n'$ better, t
In~\cite{Nakamura}, the second-named author studied the structure classes of extension~\eqref{eq:topic} for $N=PB'_n$ and $N=B_n[4]$.
He gave explicit descriptions of cocycles that represent the structure classes of these extensions.
We repeat his main result in Theorem~\ref{th:Nakamura} below.
One of the goals of the current paper is to give context to the results in~\cite{Nakamura}; further, those results are a key ingredient in our splitting results.

Certain $\sg{n}$-modules show up as subquotients of $PB_n/PB'_n$, and we investigate their low-dimensional cohomology in this paper.
In particular, these include the permutation modules and Specht modules for the partitions $(n-1,1)$ and $(n-2,2)$, working over a ring $R$ that is either the integers $\Z$ or a quotient $\Z/m\Z$.
We use the notations $M^1_R$ and $M^2_R$ for these permutation modules, and $S^1_R$ and $S^2_R$ for these Specht modules, respectively.
As a final bit of background, we summarize what is already known about the cohomology groups of $\sg{n}$ in these modules.
See section~\ref{ss:SpechtDescript} for definitions.
We study the cohomology groups $H^i(\sg{n};M)$ for $i\in\{0,1,2\}$, and where $M$ is one of the four kinds of modules just mentioned.
If $R=\Z/p\Z$ for a prime $p$, Theorems~12.1 and~24.15 of James~\cite{James} give the multiplicities of the composition factors in a composition series for $S^1_R$ and $S^2_R$, and the low-dimensional cohomology of $\sg{n}$ in these factors is given in results on page~176 of Burichenko--Kleschev--Martin~\cite{BKM}.
So in principal, these cohomology groups could be computed this way.
If $R=\Z$, then computer-assisted results of Weber~\cite{Weber} give the elementary divisors of $H^2(\sg{n};M)$, where $M$ is one of these Specht modules, up through $n=20$ (see the table in the appendix of the ArXiv version of Weber~\cite{WeberArXiv}).
These results agree with our theorems below.
We note that these families of modules are finitely generated FI-modules, and as such, the general behavior of their cohomology is predicted by Nagpal~\cite{Nagpal}.

\subsection{Results}

Our first main result is the following.

\begin{theorem}\label{th:classification}
Let $n\geq 3$.
There is a list of $8$ normal subgroups of $B_n$, such that for every normal subgroup $N\lhd B_n$ with $PB_n'\leq N\leq PB_n$, $N$ is a finite-index subgroup of exactly one subgroup in the list.
\end{theorem}

Two of these subgroups are $PB_n$ and $PB'_n$, but we believe that the other six subgroups have not been described previously.
We call these six subgroups the \emph{Specht subgroups} of the braid group (see Definition~\ref{de:Nij}).
As mentioned above, $PB_n/PB'_n$ is the permutation module $M^2_\Z$ for $(n-2,2)$-tabloids working over the integers $\Z$.
The proof of Theorem~\ref{th:classification} comes from the rigidity of submodules of this same module $M^2_\Q$, but working over the rationals $\Q$.
We define the Specht subgroups as the preimages in $PB_n$ of the kernels of the projections from $M^2_\Q$ to its irreducible factors.  
We give a more careful statement of Theorem~\ref{th:classification} in Theorem~\ref{th:classificationprecise}, and prove it there.
We also provide characterizations of the Specht subgroups in terms of winding numbers (Proposition~\ref{pr:windeqns}) and in terms of infinite generating sets (Proposition~\ref{pr:biggensets}).

Just as the Specht subgroups deserve special attention, so do the quotients of $PB_n$ by Specht subgroups.
%Since Specht subgroups are the maximal subgroups $N$ as in extension~\eqref{eq:topic}, the quotients of $PB_n$ by Specht subgroups are the minimal $\Z\sg{n}$-modules $A$ as in that extension.
In section~\ref{se:ssquot}, we investigate the quotients of $PB_n$ by two of the Specht subgroups, by determining the images of the integral scalings of two of the factor projection maps from $M^2_\Q$ to itself.
We characterize these images in Propositions~\ref{pr:f1image} and~\ref{pr:f2M2congruences}.
We find that these images are isomorphic to finite-index submodules of two integral Specht modules, but are not isomorphic to the Specht modules themselves.
See Corollaries~\ref{co:f1M2index}, \ref{co:f1M2isomorphism}, \ref{co:f2M2index}, and~\ref{co:f2M2isomorphism}.
These results are also important for our splitting results.

Our second main result is our computation of the second cohomology in two Specht modules.
This is relevant to the braid group because the modules $S^1_\Z$ and $S^2_\Z$ are submodules of $M^2_\Z$, which is isomorphic to $PB_n/PB'_n$.
\begin{theorem}\label{th:cohomcomps}
    Let $n\geq 6.$
    Let $R$ be a commutative ring, and let $S^1_R$ and $S^2_R$ be the Specht modules over $R$ for the partitions $(n-1,1)$ and $(n-2,2)$, respectively.
    We have
    \[H^2(\sg{n};S^1_R)=\left\{\begin{array}{cl}
    0 & \text{if $n$ is odd} \\
    R[2]^2\oplus R/2R & \text{if $n$ is even,} \\
    \end{array}\right.
    \]
    and
    \[
    H^2(\sg{n};S^2_R)= 
    \left\{
    \begin{array}{cl}
    R/2R & \text{if $n\equiv 0 \pmod{4}$} \\
    R[2]^2\oplus R/2R & \text{if $n\equiv 1 \pmod{4}$} \\
    R[2]\oplus R/2R & \text{if $n\equiv 2 \pmod{4}$ or $n\equiv 3 \pmod{4}$.}
    \end{array}
    \right.
    \]
\end{theorem}
These proofs use filtrations of $M^1_R$ and $M^2_R$ where the subquotients are Specht modules.  
We use explicit descriptions of cocycles to compute the maps in the resulting long exact sequences.
To give these explicit cocycles, we use a $3$-skeleton for a $K(\sg{n},1)$ that we describe in~\cite{DN3skeleton}.
To get the argument started, we compute the cohomology in $M^1_R$ and $M^2_R$ using Shapiro's lemma, since these are isomorphic to coinduced modules.
The statements in Theorem~\ref{th:cohomcomps} appear in Proposition~\ref{pr:S1cohomology} and Theorem~\ref{th:S2cohomology}, and we prove them there.
These statements also include descriptions of $H^0$ and $H^1$ in the same modules, and descriptions of $H^2$ in these modules for small $n$.

\begin{remark}
We originally computed $H^2(\sg{n};S^2_\Z)$ because we believed it would contain the structure class for an extension as in~\eqref{eq:topic}.
However, Corollary~\ref{co:f2M2isomorphism} shows that the quotient of $PB_n$ by the appropriate Specht subgroup is not $S^2_\Z$, but is rather a finite-index submodule that is not isomorphic to it.
Therefore we believe that $S^2_\Z$ does not show up as the kernel $A$ of extension~\eqref{eq:topic} for any choice of $N$.
Our splitting results do not directly quote Theorem~\ref{th:S2cohomology}, but we believe the result to be of independent interest, so we have kept it in the paper.
\end{remark}

Our final main result concerns the splitting of extensions of the form~\eqref{eq:topic} when $N$ is a Specht subgroup.
There is exactly one Specht subgroup, denoted $N_{02}$, with the property that $(PB_n/N_{02})\otimes_\Z \Q\cong S^1_\Q$.
This ends up being the only Specht subgroup where the quotient of $B_n$ by the subgroup gives a split extension.
\begin{theorem}
    Let $n\geq 4$.
    Let $N\lhd B_n$ be a Specht subgroup.
    Consider the extension
    \[0\to PB_n/N\to B_n/N \to \sg{n}\to 1.\]
    This extension splits if $N=N_{02}$ and $n$ is odd.
    It does not split otherwise.
\end{theorem}
This is part of a more general theorem, Theorem~\ref{th:QbySpechtSplitting}, and we prove it there.
Theorem~\ref{th:QbySpechtSplitting} also considers extensions where the kernel is a finite module that is the reduction of $M^2_\Z$ modulo $m$, or the result of reducing modulo $m$ and taking the quotient by a Specht module.

\subsection{Acknowledgments}
We would like to thank the following people for helpful conversations: Wade Bloomquist, Susan Hermiller, Dan Margalit, Daniel Minahan, Peter Patzt, Andrew Putman, and Nick Salter.

\section{Preliminaries for the classification}
\subsection{Braid groups and symmetric groups}\label{se:bgprelims}
For this section, fix an integer $n\geq 2$.
We recall the usual presentations for the braid group $B_n$ and the symmetric group $\sg{n}$.  
For the braid group, we use $\sigma_i$ to denote the positive half-twist on strands $i$ and $i+1$, so that $B_n$ is generated by $\sigma_1,\dotsc,\sigma_{n-1}$.
The relations are that non-consecutive generators commute and consecutive generators satisfy a braid relation $\sigma_i\sigma_{i+1}\sigma_i=\sigma_{i+1}\sigma_i\sigma_{i+1}$.
So
\[B_n=\langle \sigma_1,\dotsc,\sigma_{n-1}\,|\, \{[\sigma_i,\sigma_j]|j>i+1\},\{\sigma_i\sigma_{i+1}\sigma_i\sigma_{i+1}^{-1}\sigma_i^{-1}\sigma_{i+1}^{-1}\}_i\rangle.\]
(Our commutator convention is that $[x,y]=xyx^{-1}y^{-1}.$)

We use a similar presentation for the symmetric group.
For each $i$ with $1\leq i\leq n-1$, let $s_i$ denote the transposition $(i,i+1)$.
Then $\sg{n}$ is generated by $s_1,\dotsc, s_{n-1}$, and the relations are the same as in the braid group, but with the addition of the squaring relations (giving each generator order $2$).
So
\begin{equation}\label{eq:Snpres}
\sg{n}=\langle s_1,\dotsc,s_{n-1}\,|\, \{[s_i,s_j]|j>i+1\},\{s_i s_{i+1} s_i s_{i+1}^{-1} s_i^{-1} s_{i+1}^{-1}\}_i, \{s_i^2\}_i\rangle.
\end{equation}

Let $\rho$ be the homomorphism $B_n\to \sg{n}$ by $\sigma_i\mapsto s_i$.
The presentations make it clear that $\rho$ is well defined and surjective. 
The kernel of this map is the \emph{pure braid group} $PB_n$.
The resulting exact sequence is important for this paper:
\[1\to PB_n\to B_n\stackrel{\rho}{\longrightarrow} \sg{n}\to 1.\]
Basic facts about $PB_n$ can be found in Chapter~9 of Farb--Margalit~\cite{FM}.
For generators for $PB_n$, we use the $\binom{n}{2}$ whole twists $a_{ij}$.
Here $a_{i,i+1}=\sigma_i^2$, and if $j>i$, then 
\[a_{ij}=\sigma_{j-1}\sigma_{j}\dotsm\sigma_{i+1}\sigma_i^2\sigma_{i+1}^{-1}\dotsm\sigma_{j}^{-1}\sigma_{j-1}^{-1}.\]
Artin proved that these elements generate, and gave a presentation in~\cite{Artin}.
%Other presentations can be found in Margalit--McCammond~\cite{MMPBpres}, 
The presentation is interesting, but we will not directly use it in this work.
%We need to understand two things about $PB_n$: the abelianization and the center.

Let $\{v_{ij}\,|\,1\leq i<j\leq n\}$ be the basis for $\Z^{\binom{n}{2}}$.
The following proposition is well known (see section~9.3 of Farb--Margalit~\cite{FM}), and is an easy consequence of the presentation for $PB_n$.
We don't give a proof here.
\begin{proposition}\label{pr:PBabelianization}
The map on generators $a_{ij}\mapsto v_{ij}$ extends to a surjective homomorphism $PB_n\to \Z^{\binom{n}{2}}$, and this map is the abelianization map for $PB_n$.
\end{proposition}

We also need the center of $PB_n$.
The following is well known.
A proof appears in two parts in sections~9.2 and~9.3 of Farb--Margalit~\cite{FM}.
\begin{proposition}\label{pr:PBncenter}
    The center of $B_n$ is infinite cyclic, generated by the pure braid
    \[(a_{12}a_{13}\dotsm a_{1n})(a_{23}a_{24}\dotsm a_{2n})\dotsm(a_{n-2,n-1}a_{n-2,n})a_{n-1,n}.\]
    In particular, this cyclic subgroup is also the center of $PB_n$.
\end{proposition}

Finally, we need the following fact about the commutator subgroup $PB'_n$.
This is part of a stronger result of Bianchi~\cite{Bianchi}, but we give a short proof here.
\begin{proposition}\label{pr:H1commutator}
    Let $n\geq 3.$
    The abelianization $H_1(PB'_n)$ of the commutator subgroup $PB'_n$ of the pure braid group is not finitely generated.
\end{proposition}

\begin{proof}
    Let $f\co PB_n\to F_2$ be a surjective homomorphism to the free group $F_2$ of rank $2$.
    Cohen--Falk--Randall construct entire families of such maps in Section~3 of~\cite{CFR}, so $f$ certainly exists.
    Then the restriction $f|_{PB'_n}\co PB'_n\to F'_2$ is also surjective.
    The five-term exact sequence for the low-dimensional homology of a short exact sequence tells us that $H_1(PB'_n)$ surjects to $H_1(F'_2)$.
    But $H_1(F'_2)$ is not finitely generated since it is the first homology of an infinite-rank graph.
    So $H_1(PB'_n)$ is not finitely generated either.
\end{proof}

\subsection{Some permutation and Specht modules}
The classical Maschke's theorem implies that any finitely generated module for a finite group, over a field of characteristic $0$, decomposes uniquely as direct sum of simple modules.
Working over the rationals $\Q$, the simple $\Q \sg{n}$-modules (that is, the irreducible representations) are the \emph{Specht modules}, which correspond to partitions of $n$.
We are only interested in three Specht modules in this paper, so we do not give a careful definition for a general partition.
Our definitions work over a general commutative ring $R$, and the resulting modules are interesting, even though they are not usually simple, and usually no Maschke-style decomposition holds.
So we fix a commutative ring $R$ in this section to make our definitions.
We require commutative rings to have a multiplicative identity $1$.

We summarize the general definitions first.
We follow many of the notations and conventions from James~\cite{James}.
Let $\lambda = (\lambda_1,\dotsc,\lambda_r)$ be a partition of $n$ (with $\lambda_1\geq\dotsm\geq\lambda_r$).
The \emph{Young diagram} of $\lambda$ is the left-justified grid of boxes where the $i$th row has $\lambda_i$ boxes.
A \emph{Young tableau} is a filling of the Young diagram by the numbers $1,\dotsc, n$, with no repeats.
The \emph{tabloid} of a tableau $t$ is the set of tableaux obtained from $t$ by permuting the entries in each row (so it is a tableau with unordered rows).
The canonical action of $\sg{n}$ on $\{1,2,\dotsc,n\}$ naturally defines an action on the set of $\lambda$-tableaux, by permuting the numbers in the boxes.
This action descends to an action on the set of $\lambda$-tabloids.
The \emph{$\lambda$-permutation module} $M^\lambda_R$ is the free $R$-module on the set of $\lambda$-tabloids, made into an $R\sg{n}$-module by extending the action on tabloids linearly.
Given a tableau $t$, the \emph{polytabloid} of $t$ is the signed sum of the orbit of the tabloid of $t$ under the action of the column stabilizer of $t$.
The \emph{$\lambda$-Specht module} $S^\lambda_R$ is the submodule of $M^\lambda_R$ spanned by the set of all polytabloids.
We repeat these definitions more carefully for the three partitions we are interested in: $(n)$, $(n-1,1)$, and $(n-2,2)$.

First of all, for the trivial partition $\lambda=(n)$, we have $S^{(n)}_R=M^{(n)}_R=R$, with the untwisted action.
We use the notation $S^0_R$ for $S^{(n)}_R$.

Next we consider the modules related to the partition $\lambda=(n-1,1)$.
We use the notation $M^1_R$ as an abbreviation for $M^{(n-1,1)}_R$.
The action of $\sg{n}$ on the $(n-1,1)$-tabloids is simply the canonical action of $\sg{n}$ on $\{1,2,\dotsc, n\}$, so to work with $M^1_R$, we name the basis elements and define the action directly.
So $M^1_R$ is the free $R$-module with basis $\{t_1,\dotsc, t_n\}$,
 and it is an $R\sg{n}$-module by permuting the basis elements: for $\sigma\in \sg{n}$ and any basis element $t_i$, we have $\sigma t_i=t_{\sigma(i)}$.
 (We can think of $t_i$ as being the tabloid consisting of tableaux with $i$ as the unique entry in the second row, or we can think of it as a meaningless symbol.)

Similarly, we use $S^1_R$ as an abbreviation for $S^{(n-1,1)}_R$.
Tracing through the definitions above, we find that the polytabloids are just the elements of $M^1_R$ of the form $t_i-t_j$, for all distinct choices of $i$ and $j$ from $1$ through $n$.
So the Specht module $S^1_R$ is the span of $\{t_i-t_j\,|\,1\leq i<j\leq n\}$ in $M^1_R$.
We note that $M^1_R$ has a free basis 
\[t_1, t_2-t_1, t_3-t_1,\dotsc, t_n-t_1,\]
and from this it quickly follows that $S^1_R$ is a free $R$-module of rank $n-1$ with basis $t_2-t_1,\dotsc, t_n-t_1$.

Finally, we introduce the modules for the partition $\lambda=(n-2,2)$.
An $(n-2,2)$-tabloid is uniquely determined by the unordered pair of elements appearing in the second row of any tableau representing it.
So the action of $\sg{n}$ on the set of $(n-2,2)$-tabloids is just the action of $\sg{n}$ on the set of unordered pairs from $\{1,2,\dotsc, n\}$.
We use $M^2_R$ as notation for $M^{(n-2,2)}_R$.
Again, we name basis elements and make definitions directly.
Let $M^2_R$ be the free $R$-module whose basis is the set of $\binom{n}{2}$ symbols $\{v_{ij}\,|\,1\leq i<j\leq n\}$.
Since only the unordered pair $\{i,j\}$ matters for $v_{ij}$, we use the convention that $v_{ij}=v_{ji}$.
Under this convention, we can summarize the action of $\sg{n}$ on this basis by saying that, for $\sigma\in \sg{n}$, we have $\sigma v_{ij}=v_{\sigma(i),\sigma(j)}$.
Then $M^2_R$ is an $R\sg{n}$-module by extending this action linearly.
(Again, we can think of $v_{ij}$ as being the tabloid where any representative tableau has $\{i,j\}$ as the set of elements in its second row, or we can just think of $v_{ij}$ as being a symbol.)

The main reason that we consider $M^2_R$ in this paper is the following observation.
\begin{proposition}\label{pr:recognizepairmodule}
    In the exact sequence
    \[0\to PB_n/PB_n'\to B_n/PB_n'\to \sg{n}\to 1\]
    conjugation induces a $\Z \sg{n}$-module structure on $PB_n/PB_n'$.
    With this structure, $PB_n/PB_n'\cong M^2_\Z$ as $\Z \sg{n}$-modules.
    
    Further, for any positive integer $k$, there is a normal subgroup $N\lhd B_n$ with $PB_n'<N<PB_n$, such that the exact sequence
    \[0\to PB_n/N\to B_n/N\to \sg{n}\to 1\]
    induces a $\Z \sg{n}$-module structure on $PB_n/N$ that makes it isomorphic to $M^2_{\Z/k\Z}$ as a $\Z\sg{n}$-module.
\end{proposition}

\begin{proof}
    Here $PB_n'$ is the commutator subgroup, so that $PB_n/PB_n'$ is $\Z^{\binom{n}{2}}$ by Proposition~\ref{pr:PBabelianization}.
    As with any group extension with abelian kernel, the action of the cokernel on the kernel is by lifting and conjugating (a different choice of lifts will result in the same action since the differences are in the abelian kernel).
    We lift a generator $s_i$ in $\sg{n}$ to $\sigma_i$ in $B_n$, and conjugate a generator $a_{jk}$ by this lift.
    There are a few cases to consider, but in all cases the conjugate $\sigma_ia_{jk}\sigma_i^{-1}$ can be rewritten (using the definition of $a_{jk}$ and the commutation and braid relations from Section~\ref{se:bgprelims}) to be $a_{s_i(j),s_i(k)}$.
    This matches the action on $M^2_\Z$, proving the first statement.

    For the second statement, simply define $N$ to be the subgroup of $PB_n$ generated by $PB_n'$ and 
    $\{a_{ij}^k\,|\,1\leq i<j\leq n\}$.  
    The isomorphism of $PB_n/PB_n'$ to $M^2_\Z$ will carry $N/PB_n'$ to $kM^2_\Z$, so that $PB_n/N$ is isomorphic to $M^2_\Z/kM^2_\Z\cong M^2_{\Z/k\Z}$ as a $\Z \sg{n}$-module.
\end{proof}

Now we consider the Specht module in $M^2_R$.
Suppose $i,j,k,l$ are distinct numbers from $1$ through $n$, and $t$ is an $(n-2,2)$-tableau with $i,j,k,l$ as the entries in the left $2\times 2$-square (with other entries unspecified):
\[
t=
\begin{ytableau}
i & j & \ast & \none[\dots] & \ast \\
k & l 
\end{ytableau}
\]
Then the polytabloid of $t$ in $M^2_R$ is $v_{ij}+v_{kl}-v_{il}-v_{kj}$.
So the Specht module $S^2_R=S^{(n-2,2)}_R$ is the submodule of $M^2_R$ spanned by 
\[\{v_{ij}+v_{kl}-v_{il}-v_{kj}\,|\,\text{$i,j,k,l$ distinct in $1,2,\dotsc,n$}\}.\]
In the next section, we find a free basis for $S^2_R$ as an $R$-module.

\subsection{Structure of $M^2_R$}\label{ss:structure}
First we describe a different basis for $M^2_R$.
A \emph{standard tableau} is one in which the numbers are strictly increasing in every row and every column.
For $(n-2,2)$, a standard tableau is one with a $1$ in the upper left, any increasing ordered pair other than $(2,3)$ in the second row, and the remaining elements in increasing order in the rest of the first row.
This means that $3$ is in the second position of the first row if $2$ is in the second row, and $2$ is in the second position of the first row, otherwise.
We use the notation $e_{ij}$ for the polytabloid of standard tableau with $(i,j)$ as its second row:
\[e_{2i}=v_{2i}+v_{13}-v_{1i}-v_{23}, \text{ for $4\leq i\leq n$ and}\]
\[e_{ij}=v_{ij}+v_{12}-v_{1j}-v_{2i}, \text{ for $3\leq i<j\leq n$.}\]
Together, these form the set of \emph{standard polytabloids} in $M^2_R$.
The number of such elements is $\binom{n}{2}-n$ (the $n$ unordered pairs that are not allowed are the $n-1$ pairs involving $1$, and the pair $\{2,3\}$).
\begin{lemma}\label{le:M2basis}
    Let $n\geq 4$.
    Then $M^2_R$ has a free basis over $R$ consisting of $v_{12}$, the elements $v_{1j}-v_{12}$ for $3\leq j\leq n$, the element $v_{23}-v_{12}$, and the standard polytabloids $e_{24},\dotsc, e_{2n}$, $e_{34},\dotsc,e_{n-1,n}$.
\end{lemma}
\begin{proof}
    We order the set of unordered pairs $\{i,j\}$ with $1\leq i<j\leq n$ in the usual way: $\{i,j\}< \{k,l\}$ if $\min\{i,j\}<\min\{k,l\}$, or if $\min\{i,j\}=\min\{k,l\}$ and $\max\{i,j\}<\max\{k,l\}$.  
    
    Let $A$ be the matrix whose columns are the list of vectors 
 in the statement, in order, written as column vectors with respect to the standard basis for $M^2_R$ in the order just described.
    Then $A$ is an upper-triangular matrix with entries of $1$ on the diagonal, since all the other summands of $e_{ij}$ come earlier than $v_{ij}$ in the standard order.
    So $A$ is a product of elementary matrices over $R$, meaning that the set of vectors in the statement comes from the standard basis by a sequence of transvections and is therefore a basis.
\end{proof}

There is an $R\sg{n}$-module homomorphism $M^2_R\to S^0_R$ by $v_{ij}\mapsto 1$.
Let $K_{12}$ denote the kernel of this map.
Since the map is surjective, this definition gives us an exact sequence
\[0\to K_{12}\to M^2_R\to S^0_R\to 0.\]
Let $\mu\co M^2_R\to M^1_R$ be the $R\sg{n}$-module homomorphism defined by $\mu(v_{ij})=t_i+t_j$.
\begin{proposition}\label{pr:SESS2K12S1}
    For $n\geq 4$, we have an exact sequence
    \[0\to S^2_R\to K_{12}\stackrel{\mu|_{K_{12}}}{\longrightarrow} S^1_R\to 0.\]
\end{proposition}
\begin{proof}
The basis from Lemma~\ref{le:M2basis} has $v_{12}$ mapping to $1$ in $S^0_R$, and all other basis elements mapping to $0$.
From this it follows that $K_{12}$ is a free $R$-module with basis given by the remaining $\binom{n}{2}-1$ elements of that basis for $M^2_R$.
For any generator $t_i-t_j$ of $S^1_R$, we can pick $k$ distinct from $i$ and $j$, and $\mu(v_{ik}-v_{jk})=t_i-t_j$.
So $\mu|_{K_{12}}$ maps surjectively to $S^1_R$.
The vectors $v_{23}-v_{12}$ and all the $v_{1j}-v_{12}$ together map to a linearly independent set of vectors in $S^1_R$.
(To see this, note that that $\mu(v_{23}-v_{12})=t_3-t_1$ is the only such vector to have a nonzero coefficient on $t_1$, and for each $j>3$, $\mu(v_{1j}-v_{12})=t_j-t_2$ is the only vector to have a nonzero coefficient on $t_j$.)
From this it follows that $\ker(\mu|_{K_{12}})$ is freely generated by the standard polytabloids, and therefore that it is a submodule of $S^2_R$.
On the other hand, any polytabloid is in $K_{12}$, and maps to $0$ under $\mu$.
So $S^2_R$ is contained in $\ker(\mu|_{K_{12}})$.
So they are equal, and the sequence is exact.
\end{proof}

\begin{corollary}\label{co:S2gens}
For $n\geq 4$, $S^2_R$ is a free $R$-module of rank $\binom{n}{2}-n$, freely generated by the standard polytabloids.
\end{corollary}

Now we observe that we can find copies of $S^0_R$ and $S^1_R$ in $M^2_R$.
To do this, we define some extra terms in $M^2_R$.
Define
\[u=\sum_{1\leq i<j\leq n} v_{ij}\in M^2_R, \text{ and for all $i$, } w_i=\sum_{\substack{1\leq j\leq n\\ j\neq i}} v_{ij}.\]
Then the maps $S^0_R\to M^2_R$ by $1\mapsto u$ and $S^1_R\to M^2_R$ by $(t_i-t_j)\mapsto (w_i-w_j)$ are injective $R\sg{n}$-module homomorphisms.

We need some maps the are defined over the integers, and more generally.
These are built from maps that arise over $\Q$, so we temporarily specialize to $R=\Q$.
From the fact that the $S^i_\Q$ are irreducible, and using dimension-counting, we can deduce that:
\[M^1_\Q\cong S^1_\Q\oplus S^0_\Q,\quad\text{and}\quad M^2_\Q\cong S^2_\Q\oplus S^1_\Q\oplus S^0_\Q.\]
More constructively, we can find formulas for the idempotent projection endomorphisms $\pi^i\co M^2_\Q\to M^2_\Q$ for $i\in \{0,1,2\}$, with the properties that $\pi^i(M^2_\Q)\cong S^i_\Q$ and $\pi^i\circ \pi^i=\pi^i$ for each $i$, and $\pi^i\circ \pi^j=0$ if $i\neq j$, and $\pi^0+\pi^1+\pi^0=\mathrm{id}$.
These formulas are:
\[\pi^0(v_{ij})=\frac{2}{n(n-1)}u,\quad
\pi^1(v_{ij})=\frac{1}{n-2}(w_i+w_j)-\frac{4}{n(n-2)}u,\text{ and }\]
\[\pi^2(v_{ij})=v_{ij}-\frac{1}{n-2}(w_i+w_j)+\frac{2}{(n-1)(n-2)}u.\]
To see that these formulas are correct, it is an exercise to check that $\pi^0$ satisfies $u\mapsto u$, $w_i-w_j\mapsto 0$, and $e_{ij}\mapsto 0$, and $\pi^1$ satisfies $u\mapsto 0$, $w_i-w_j\mapsto w_i-w_j$, and $e_{ij}\mapsto 0$, and $\pi^2$ satisfies $u\mapsto 0$, $w_i-w_j\mapsto 0$, and $e_{ij}\mapsto e_{ij}$, for all appropriate $i$ and $j$.

%We also name the kernels of these maps.
%\begin{definition}
%    For each $i\in \{0,1,2\}$, let $K_i=\ker(f^i)\subset M^2_R$.

%    For distinct $i,j\in\{0,1,2\}$, let $K_{ij}=K_i\cap K_j$.

%\end{definition}

%Notice that this matches the definition of $K_{12}$ that we gave earlier.
%We are now ready to discuss the subgroups in Theorem~\ref{th:classification}.

\section{Specht subgroups of $B_n$}
\subsection{Definition and first properties}
We use the notation of the last section, but work over $\Z$.
In this section, we have a standing assumption that $n\geq 4$.
Let $\pi\co PB_n\to M^2_\Z$ be the abelianization map.
\begin{definition}\label{de:Nij}
Let $i,j$ be distinct elements of $\{0,1,2\}$ and let $k\in\{0,1,2\}$ be the remaining element.
Define $N_{ij}=\ker(\pi^k|_{M^2_\Z}\circ \pi)<PB_n$.
Now let $i\in\{0,1,2\}$ and let $j,k\in\{0,1,2\}$ be the other two elements.
Define $N_i=N_{ij}\cap N_{ik}$.
%For each $i\in \{0,1,2\}$, define $N_i=\pi^{-1}(K_i)\leq PB_n$, and for distinct $i,j\in \{0,1,2\}$, define $N_{ij}=\pi^{-1}(K_{ij})\leq PB_n$.
Call these six subgroups the \emph{Specht subgroups} of $B_n$.
\end{definition}
The intuition is that $N_i$ ``contains" $S^i_\Z$, and $N_{ij}$ ``contains" both $S^i_\Z$ and $S^j_\Z$, in a sense that we formalize in Lemma~\ref{le:tensorwithQ} below.
Since this definition is in terms of the $\Z \sg{n}$-module homomorphisms defined in the previous section, it might be hard to follow.
We provide characterizations of each subgroup in terms of winding numbers and in terms of generating sets in subsection~\ref{ss:SpechtDescript} below.

First we note that these subgroups really are normal subgroups of $B_n$, as a consequence of the following proposition.
\begin{proposition}\label{pr:normaliffsubmodule}
    Let $N$ be a subgroup of $B_n$ with $PB_n'\leq N\leq PB_n$.
    Then $N$ is a normal subgroup of $B_n$ if and only if $\pi(N)$ is a is a submodule of the $\Z\sg{n}$-module $M^2_\Z$.
\end{proposition}

\begin{proof}
    Suppose $N\leq B_n$.
    Let $\nu\in N$ and let $\gamma\in B_n$.
    Since $\pi$ and $\rho$ are surjections, $\pi(\nu)$ and $\rho(\gamma)$ are generic elements of $\pi(N)$ and $\sg{n}$, respectively.
    The definition of the action of $\sg{n}$ on $PB_n/PB_n'$ implies that $\rho(\gamma)\cdot \pi(\nu)=\pi(\gamma\nu\gamma^{-1})$.
    If $N\lhd B_n$, then for all such $\nu$ and $\gamma$, we have $\gamma\nu\gamma^{-1}\in N$, so that $\rho(\gamma)\cdot \pi(\nu)\in \pi(N)$.
    Conversely, if $\pi(N)$ is a submodule, then for all such $\nu$ and $\gamma$, we have $\rho(\gamma)\cdot \pi(\nu)\in \pi(N)$, so $\pi(\gamma\nu\gamma^{-1})\in \pi(N)$, so $\gamma\nu\gamma^{-1}\in \ker(\pi)=N.$
\end{proof}

The point of the Definition~\ref{de:Nij} is the following lemma.
This is also the meaning of the notation in that definition.
\begin{lemma}\label{le:tensorwithQ}
The quotient of any $N_i$ or $N_{ij}$ by $PB_n'$ tensors up to the indicated combination of Specht modules over $\Q$:
\[(N_i/PB_n')\otimes_\Z \Q\cong S^i_\Q\quad\text{and}\quad (N_{ij}/PB_n')\otimes_\Z \Q\cong S^i_\Q\oplus S^j_\Q.\]
Likewise, the quotient of $PB_n$ by any $N_i$ or $N_{ij}$ tensors up to the complementary combination of Specht modules.
This means that, if $(i,j,k)$ is any ordering of $\{0,1,2\}$, then 
\[(PB_n/N_i)\otimes_\Z \Q\cong S^j_\Q\oplus S^k_\Q\quad\text{and}\quad (PB_n/N_{ij})\otimes_\Z \Q\cong S^k_\Q.\]
\end{lemma}
The point of Theorem~\ref{th:classification} from the introduction is that the Specht subgroups are the maximal normal subgroups of $B_n$ that satisfy the conclusions of this lemma.
\begin{proof}[Proof of Lemma~\ref{le:tensorwithQ}]
    Fix $i,j,k$ an ordering of $\{0,1,2\}$.
    We know $\pi$ induces an isomorphism $PB_n/N_{ij}\cong M^2_\Z/\ker(\pi^k|_{M^2_\Z})$, 
    which is isomorphic to $\pi^k(M^2_\Z)$.
    %But over $\Q$, $f^k$ is essentially the projection $\pi^k$ scaled by an integer, so
    
    \[(PB_n/N_{ij})\otimes \Q = \pi^k(M^2_\Z)\otimes \Q \cong \pi^k(M^2_\Q)\cong S^k_\Q.\]
    Since $\Q$ is flat, tensoring with $\Q$ preserves exact sequences, so 
    \[0\to \ker(\pi^k|_{M^2_\Z})\to M^2_\Z \to \pi^k(M^2_\Z)\to 0\] 
    becomes
    \[0\to \ker(\pi^k|_{M^2_\Z})\otimes \Q \to S^0_\Q\oplus S^1_\Q\oplus S^2_\Q\to S^k_\Q\to 0.\]
    But $\ker(\pi^k|_{M^2_\Z})\cong N_{ij}/PB_n'$, and we deduce that $(N_{ij}/PB_n')\otimes_\Z \Q$ is isomorphic to $S^i_\Q\oplus S^j_\Q$.

    Similarly, $\pi$ induces an isomorphism $PB_n/N_{i}\cong M^2_\Z/(\ker(\pi^j|_{M^2_\Z})\cap\ker(\pi^k|_{M^2_\Z}))$.
    This is isomorphic  to $M^2_\Z/\ker(\pi^j\oplus \pi^k)$, where $\pi^j\oplus \pi^k\co M^2_\Z\to M^j_\Q\oplus M^k_\Q$.
    So $PB_n/N_{i}$ is isomorphic to $\pi^j(M^2_\Z)\oplus \pi^k(M^2_\Z)$.
    By the same argument as above,
    \[(PB_n/N_i)\otimes \Q\cong \pi^j(M^2_\Q)\oplus \pi^k(M^2_\Q) \cong S^j_\Q\oplus S^k_\Q.\]
    Then by the same argument using flatness of $\Q$, we have that $(N_{i}/PB_n')\otimes \Q$ is $S^i_\Q$.
\end{proof}

\begin{lemma}\label{le:pairwiseincomm}
    In the list of $8$ subgroups $PB_n, N_0, N_1, N_2, N_{01}, N_{02}, N_{12}, PB_n'$, the subgroups are pairwise incommensurable; in other words, the intersection of any two distinct subgroups in the list has infinite index in at least one of the subgroups.
\end{lemma}

\begin{proof}
    Suppose subgroups $N$ and $M$ from the list are commensurable.
    Then $(N/PB_n')\otimes \Q$ and $(M/PB_n')\otimes \Q$ would be isomorphic.
    Lemma~\ref{le:tensorwithQ} shows that this is not the case (together with the observations that $(PB_n/PB_n')\otimes \Q\cong S^0_\Q\oplus S^1_\Q \oplus S^2_\Q$ and $(PB_n'/PB_n')\otimes \Q=0$).
\end{proof}

\subsection{Proof of the classification}
The following is Theorem~\ref{th:classification} from the introduction.
We give a more precise statement here.
\begin{theorem}\label{th:classificationprecise}
    Suppose $N$ is a normal subgroup of $B_n$ with $PB_n'\leq N\leq PB_n$.
    Then $N$ is a finite-index subgroup of exactly one  of:
    \[PB_n, N_0, N_1, N_2, N_{01}, N_{02}, N_{12},\text{ or } PB_n'.\]
\end{theorem}

\begin{proof}
    Proposition~\ref{pr:normaliffsubmodule} implies that $\pi(N)$ is a submodule of $M^2_\Z$.
    Then $\pi$ induces an isomorphism between  $PB_n/N$ and $M^2_\Z/\pi(N)$ that allows us to recognize $PB_n/N$ as a $\Z \sg{n}$-module.
    Let $Q$ denote $M^2_\Z/\pi(N)$, and let $Q_\Q$ denote $(M^2_\Z/\pi(N))\otimes \Q$.
    We have the following composition of maps:
    \[PB_n\stackrel{\pi}{\longrightarrow} M^2_\Z \to Q\to Q_\Q.\]
    Define $K\leq PB_n$ to be the kernel of this composition.
    Since $Q$ is a finitely generated abelian group, the kernel of $Q\to Q_\Q$ has finite index in $Q$, and therefore $N$ has finite index in $K$.
    We claim that $K$ is equal to one of the eight subgroups in the list.

    Since tensoring with $\Q$ is a right-exact functor, we have that $Q_\Q$ is a quotient of $M^2_\Q$.
    Since $M^2_\Q\cong S^0_\Q\oplus S^1_\Q \oplus S^2_\Q$, there are eight possibilities for $Q_\Q$, depending on which of the three simple modules survive to the quotient.
    First we point out that if $Q_\Q=0$, then $K=PB_n$, and if $Q_\Q=M^2_\Q$, then $K=PB_n'$.

    Suppose we are in one of the three cases where $Q_\Q\cong S^i_\Q$ for $i\in \{0,1,2\}$.
    So $K$ is the kernel of a map $PB_n\to M^2_\Z\to S^i_\Q$.
    The map $M^2_\Z\to S^i_\Q$ factors through the inclusion $M^2_\Z\to M^2_\Q$.
    The other factor is a $\Q\sg{n}$-module homomorphism $M^2_\Q\to S^i_\Q$.
    Since $M^2_\Q$ is a direct sum of three simple modules, Schur's lemma tells us that this map restricts to $0$ on two of the factors, and restricts to an isomorphism on $S^i_\Q$.
    Let $\psi$ denote this isomorphism.
    Then we have the following commutative diagram:
    \[
    \xymatrix{
    M^2_\Z \ar@{^{(}->}[d] \ar[r] & Q_\Q \ar@{^{(}->}[r] & M^i_\Q \\
    M^2_\Q \ar[r]^{\pi^i} & S^i_\Q \ar[r]^\psi & S^i_\Q. \ar@{^{(}->}[u]
    }
    \]
    So $K$ is the kernel of $\psi\circ \pi^i\circ \pi$.
    But since $\psi$ is an automorphism, it doesn't change the kernel.
    %With respect to the standard basis, $\pi^i$ has some coefficients in $\Q$ that are not in $\Z$.
    %By scaling, we can clear the denominators, and absorb the scaling factor into $\psi$ to form a new automorphism $\psi'$ of $S^i_\Q$.
    %But $f^i\co M^2_\Z\to M^i_\Z$ is $\pi^i$ with denominators cleared and with the codomain altered using an embedding $M^i_\Z\to M^2_\Q$.
    %So $\psi\circ \pi^i$ has the same kernel as $\psi'\circ f^i$, and since $\psi'$ is an automorphism, this is the same as the kernel of $f^i$.
    Letting $j$ and $k$ denote the complementary indices as usual, $N_{jk}$ is the kernel of $\pi^i|_{M^2_\Z}\circ \pi$ and this means that $K=N_{jk}$ in this case.

    In the remaining cases, $Q_\Q\cong S^i_\Q\oplus S^j_\Q$ for some distinct $i,j\in \{0,1,2\}$.
    The same argument works, replacing $\pi^i\colon M^2_\Q\to S^i_\Q$ with $\pi^i\oplus \pi^j\colon M^2_\Q\to S^i_\Q\oplus S^j_\Q$.
    %, and replacing $f^i\colon M^2_\Z\to M^i_\Z$ with $f^i\oplus f^j\colon M^2_\Z\to M^i_\Z\oplus M^j_\Z$.
    The argument then tells us that $K=N_{jk}\cap N_{ik}$, which is $N_{k}$ (where $k\in\{0,1,2\}$ is the complementary element).
\end{proof}

\subsection{Descriptions of Specht subgroups}
\label{ss:SpechtDescript}
We defined the groups $N_i$ and $N_{ij}$ as the kernels of certain maps from $PB_n$ to free abelian groups.
We start this subsection by giving more concrete descriptions of these subgroups in terms of winding of pure braids.
\begin{definition}
    For an unordered pair $i,j$ from $\{1,2,\dotsc,n\}$, let $\omega_{ij}\colon PB_n\to \Z$ be the winding number of the $i$th strand around the $j$th strand in the positive direction.
\end{definition}
Under our conventions, for $\gamma\in PB_n$, the integer $\omega_{ij}(\gamma)$ is the coefficient on $v_{ij}$ of $\pi(\gamma)\in M^2_\Z$:
\[\pi(\gamma)=\sum_{i,j}\omega_{ij}(\gamma)v_{ij}.\]
\begin{proposition}\label{pr:windeqns}
    Each of the six Specht subgroups can be characterized by systems of equations in the $\omega_{ij}$ maps, as follows:
    \begin{itemize}
        \item $N_0< PB_n$ is equal to the set of $\gamma\in PB_n$ such that, for all unordered pairs $\{i,j\}$ and $\{k,l\}$ from $\{1,2,\dotsc,n\}$ (possibly with $\{i,j\}\cap\{k,l\}\neq \varnothing$), we have
        \[\omega_{ij}(\gamma)=\omega_{kl}(\gamma).\]
        \item $N_1<PB_n$ is equal to the set of $\gamma\in PB_n$ such that, for all unordered pairs $\{i,j\}$ from $\{1,2,\dotsc,n\}$, we have
        \[(n-4)\omega_{ij}(\gamma)=\sum_{k\notin\{i,j\}}(\omega_{ik}(\gamma)+\omega_{jk}(\gamma)).\]
        \item $N_2<PB_n$ is equal to the set of $\gamma\in PB_n$ such that, for all $i$ in $\{1,2,\dotsc, n\}$, 
        \[\sum_{j\neq i}\omega_{ij}(\gamma)=0.\]
        \item $N_{01}< PB_n$ is equal to the set of $\gamma\in PB_n$ such that for all unordered pairs $\{i,j\}$ from $\{1,2,\dotsc,n\}$, we have
        \[(n-2)(n-3)\omega_{ij}(\gamma)+\sum_{\{k,l\}\cap\{i,j\}=\varnothing}2\omega_{kl}(\gamma)=\sum_{k\notin\{i,j\}}(n-3)(\omega_{ik}(\gamma)+\omega_{jk}(\gamma)).\]
        \item $N_{02}<PB_n$ is equal to the set of $\gamma\in PB_n$ such that, for each $i$ from $\{1,2,\dotsc,n\}$, we have
        \[\sum_{j\neq i}(n-2)\omega_{ij}(\gamma)=\sum_{i\notin\{j,k\}}2\omega_{jk}(\gamma).\]
        \item $N_{12}<PB_n$ is equal to the set of $\gamma\in PB_n$ such that 
        \[\sum_{i,j}\omega_{ij}(\gamma)=0.\]
    \end{itemize}
\end{proposition}

\begin{proof}
    First we handle $N_{12}=\ker(\pi^0\circ \pi).$  Let $v=\sum_{ij} c_{ij}v_{ij}\in M^2_\Z$ for some coefficients $c_{ij}\in \Z$.
    Then $\pi^0(v)$ is $\sum_{ij} c_{ij}$, up to scaling, so $v\in \ker(\pi^0)$ if and only if $\sum_{ij}c_{ij}=0$.
    Set $v=\pi(\gamma)$ for $\gamma\in PB_n$, and we recover the statement.

    Next consider $N_{02}=\ker(\pi^1\circ \pi)$.
    With $v$ as above, we compute that $\pi^1(v)$ is a scalar multiple of
    \[\sum_{ij}c_{ij}\big(n(t_i+t_j)-2\sum_kt_k\big).\]
    We gather terms together to get
    \[\sum_{i=1}^n\big(\sum_{j\neq i}(n-2)c_{ij}-2\sum_{i\notin\{j,k\}}c_{jk}\big)t_i.\]
    Setting $v=\pi(\gamma)$ and setting $\pi^1(v)=0$ gives us $n$ equations, one for each $t_i$ coefficient, and these are the $n$ equations in the statement.

    We apply a similar argument for $N_{01}=\ker(\pi^2\circ \pi)$.
    This is a more substantial exercise, but with $v$ as above, gathering terms for $\pi^2(v)$, and clearing denominators, we get
    \[\sum_{ij}\Big((n-2)(n-3)c_{ij}-\sum_{k\notin\{i,j\}}(n-3)(c_{ik}+c_{jk})+\sum_{\{k,l\}\cap\{i,j\}=\varnothing}2c_{kl}\Big)v_{ij}.\]
    Again setting $v=\pi(\gamma)$ and $\pi^2(v)=0$, we get the $\binom{n}{2}$ equations in the statement.

    The remaining three cases are simpler because we use a trick.
    For a moment, we work over $\Q$.
   Let $i,j$ in $\{0,1,2\}$ be distinct and let $p$ and $q$ be nonzero rational numbers.  
   Then $\ker(p\pi^i+q\pi^j)$ is a proper submodule of $\ker(\pi^i)$, and contains $\ker(\pi^i)\cap\ker(\pi^j)$.
   Since $M^2_\Q$ is a direct sum of three simple modules, there are no submodules properly between $\ker(\pi^i)$ and $\ker(\pi^i)\cap\ker(\pi^j)$ and therefore
   \[\ker(p\pi^i+q\pi^j)=\ker(\pi^i)\cap\ker(\pi^j).\]
   This remains true if we scale these maps to clear denominators and intersect with $M^2_\Z$ to work over $\Z$.

    By the remarks in the previous paragraph, $N_0$ is the kernel of the composition of $\pi$ with the map $M^2_\Z\to M^2_\Z$ given by
    \[v_{ij}\mapsto \binom{n}{2}v_{ij}+u,\]
    since this map is a linear combination of $\pi^1$ and $\pi^2$ with nonzero coefficients.
    As in the preceding cases, we evaluate this map on $v$, gather terms, set $v=\pi(\gamma)$ and assume that $v$ is in the kernel.
    We get the following system of equations: for each pair $\{i,j\}$, we have
    \[(\binom{n}{2}-1)\omega_{ij}(\gamma)=\sum_{\{k,l\}\neq\{i,j\}}\omega_{kl}(\gamma).\]
    This system of equations is equivalent to the one in the statement, which is that for all pairs $\{i,j\}$ and $\{k,l\}$, we have $\omega_{ij}(\gamma)=\omega_{kl}(\gamma)$.

    For $N_1$, we recognize $N_1$ as the kernel of the composition of $\pi$ with the map $M^2_\Z\to M^2_\Z$ given by
    \[v_{ij}\mapsto (n-2)v_{ij}-(w_i+w_j),\]
    since this map is a linear combination of $\pi^0$ and $\pi^2$ with nonzero coefficients.
    We set $v=\pi(\gamma)$ in the kernel, gather terms, and harvest equations from the coefficients.
    We get the system in the statement.

    Finally, for $N_2$, we recognize $N_2$ as the kernel of the composition of $\pi$ with the map $M^2_\Z\to M^1_\Z$ given by
    \[v_{ij}\mapsto t_i+t_j.\]
    This is a linear combination of $\pi^1$ and $\pi^0$ with nonzero coefficients, with the codomain altered using the embedding of $M^1_\Z\to M^2_\Z$ by $t_i\mapsto w_i$.
    Again, we set $v=\pi(\gamma)$ in the kernel, gather terms, harvest equations, and get the system in the statement.
\end{proof}

Next we describe infinite generating sets for each Specht subgroup.
We name some elements that we will reuse.
Let $z\in PB_n$ be the generator for the center of $B_n$ given in Proposition~\ref{pr:PBncenter}.
%\[z=a_{12}a_{13}\dotsm a_{1n}a_{23}a_{24}\dotsm a_{2n}\dotsm a_{n,n-1}.\]
For each $i$ from $1$ to $n$, the element $y_i$ is the element given by dragging the $i$th strand around the other strands positively:
\[y_i=a_{1i}a_{2i}\dotsm a_{i-1,i}a_{i,i+1}\dotsm a_{i,n}.\]
The \emph{lifts of the standard polytabloids} are the following $\binom{n}{2}-n$ elements of $PB_n$:
\[\{a_{13}a_{2i}a_{1i}^{-1}a_{23}^{-1}\,|\,4\leq i\leq n\}
\cup \{a_{12}a_{ij}a_{1j}^{-1}a_{2i}^{-1}\,|\,3\leq i<j\leq n\}.\]

\begin{remark}
    We don't yet have a good understanding of which of the Specht subgroups are finitely generated.
    Proposition~\ref{pr:biggensets} only  proves that, if $N$ is a Specht subgroup, then $N/PB_n'$ is finitely generated.  
    By Proposition~\ref{pr:H1commutator}, $PB_n'$ is not finitely generated, so finite generation of these subgroups does not follow easily.
    Two of the subgroups have easy answers, and we explain these in Propositions~\ref{pr:N12fg} and~\ref{pr:N0nfg} below.
    %Preliminary investigation suggests that $N_{12}$ is finitely generated, but the others may not be.
\end{remark}

\begin{proposition}\label{pr:biggensets}
    Each of the six Specht subgroup is generated by the union of $PB'_n$ with a finite set of pure braids, as follows:
    \begin{itemize}
        \item $N_0$ is generated by the union of $PB'_n$ with the single element $z$.
        \item If $n$ is odd, then $N_1$ is generated by the union of $PB'_n$ with the $n-1$ elements
        \[\{y_iy_1^{-1}\,|\,2\leq i\leq n\};\]
        If $n$ is even, then $N_1$ is generated by the union of the above generators with the element $zy_1^{-n/2}$.
        \item $N_2$ is generated by the union of $PB'_n$ with the $\binom{n}{2}-n$ lifts of standard polytabloids.
        \item $N_{01}$ is generated by the union of $PB'_n$ with the $n$ elements
        \[\{z,y_1,y_2,\dotsc,y_{n-1}\}\]
        \item $N_{02}$ is generated by the union of $PB'_n$ with the $\binom{n}{2}-n+1$ elements given by the lifts of standard polytabloids, and $z$.
        \item $N_{12}$ is generated by the union of $PB'_n$ with the $\binom{n}{2}-1$ elements
        \[\{a_{ij}a_{12}^{-1}\,|\,1\leq i<j\leq n, \{i,j\}\neq \{1,2\}\}\]
    \end{itemize}
\end{proposition}

\begin{proof}
    First consider $N_{01}$.
    Let $K_{01}=\ker(\pi^2|_{M^2_\Z})$; then $N_{01}=\pi^{-1}(K_{01})$.
    We know that 
    %$\pi^2(e_{ij})=e_{ij}$ for each standard polytabloid, and 
    the image $\pi^2(M^2_\Z)$ is a lattice in $S^2_\Q$ of rank $\binom{n}{2}-n$.
    We choose a basis for $M^2_\Z$ to find the kernel.
    In the standard basis for $M^2_\Z$, replace each $v_{in}$ with $w_i$, for $i=1$ through $n-1$, and replace $v_{12}$ with $u$. 
    This is a basis because each replacement replaces a vector by plus or minus itself, plus a sum of vectors in the basis at that step in the replacement process.
    Each $w_i=v_{i,n}+\sum_{j=1}^{n-i}v_{ij}$, and 
    \[u=-v_{12}+\sum_{i=1}^{n-1}w_i-\sum_{\substack{1\leq i<j\leq n-1\\ \{i,j\}\neq\{1,2\}}}v_{ij}.\]
    %In the standard ordering for the standard basis, the change-of-basis matrix for the new basis is upper-triangular with ones on the diagonal, and so it is  clear that the new basis is a basis.
    However, direct computation shows that $\pi^2$ sends each of $w_1,\dotsc, w_{n-1}$ and $u$  to $0$.
    Since there are $\binom{n}{2}-n$ remaining elements in the basis and this equals the rank of the image, it follows that the remaining basis elements map to linearly independent vectors, and that  $w_1,\dotsc, w_{n-1},u$ form a basis for $K_{01}$.
    The statement follows.

    Next consider $N_{12}$.
    This is $\pi^{-1}(K_{12})$, where $K_{12}$ is the kernel of $\pi^0|_{M^2_\Z}$ as in section~\ref{ss:structure}.
    Define a basis for $M^2_\Z$ by altering the standard basis: replace each $v_{ij}$ with $v_{ij}-v_{12}$, except for $v_{12}$, which we leave the same.
    Since $\pi^0$ sends $v_{12}$ to $1$ and sends the other basis vectors to $0$, these other basis vectors form a free basis for $K_{12}$.
    The statement about generators for $N_{12}$ follows.

    To find generators for $N_1$, we restrict $\pi^0$ to $K_{01}$.
    If we define $K_1=\ker(\pi^0|_{K_{01}})$, then $N_1=\pi^{-1}(K_1)$.
    We use the basis $w_1,\dotsc, w_{n-1}, u$ for $K_{01}$  that we just found.
    For each $i$, $\pi^0(w_i)=n-1$, and $\pi^0(u)=\binom{n}{2}$.
    Assume for a moment that $n$ is odd.
    Then $K_1$ is generated by the differences $w_i-w_j$, and $2u-nw_i$ (since $2\binom{n}{2}$ is the least common multiple of $\binom{n}{2}$ and $n-1$, if $n$ is odd).
    But $2u=\sum_{i=1}^nw_i$, so all the generators in this generating set are sums of the generators $w_2-w_1,\dotsc,w_n-w_1$.
    The statement about generators for $N_1$ follows, when $n$ is odd.
    The argument when $n$ is even is similar, except that now the generators are the $w_i-w_j$ and $u-(n/2)w_i$.
    We cannot express $u$ as a $\Z$-linear combination of the $w_i$, so our generating set is $w_2-w_1,\dotsc,w_n-w_1$, and $u-(n/2)w_1$ (in fact, we can eliminate $w_n-w_1$ from this set, but not $u-(n/2)w_1$).
    The statement about generators for $N_1$ follows.
    
    Now consider $N_0$.
    By Proposition~\ref{pr:windeqns}, this is the subgroup of $PB_n$ consisting of braids with all winding numbers equal.
    Suppose $\gamma\in N_0$ and write $\gamma$ as a product of whole twists $a_{ij}$ (since $\gamma\in PB_n$).
    Since $\gamma\in N_0$, there is some $k\in \Z$ such that $\omega_{ij}(\gamma)=k$ for all $i$ and $j$.
    Then modulo $PB_n'$, we have that $\gamma=z^k$.

    The statement for $N_2$ follows by noting that $N_2=\pi^{-1}(S^2_\Z)$, and quoting Corollary~\ref{co:S2gens}, which says that $S^2_\Z$ is generated by the standard polytabloids.
    The statement for $N_{02}$ then follows since $N_{02}$ is generated by $N_0$ and $N_2$.
\end{proof}

The following proposition came out of conversations with Dan Margalit.
\begin{proposition}\label{pr:N12fg}
    The Specht subgroup $N_{12}$ is finitely generated.
\end{proposition}

\begin{proof}
    By Proposition~\ref{pr:windeqns}, $N_{12}$ is the kernel of the map $PB_n\to \Z$ by $\gamma\mapsto \sum_{i,j}\omega_{ij}(\gamma)$.
    Define $\widetilde N_{12}$ to be the kernel of the map $PB_n\to \Z/\binom{n}{2}\Z$ by the same formula.
    Then $\widetilde N_{12}$ is finitely generated since it is a finite-index subgroup of the finitely generated group $PB_n$.
    Let $q\co \widetilde N_{12}\to \widetilde N_{12}/N_{12}\cong \Z$ be the quotient map.  
    As above, let $z\in PB_n$ be the given generator for the center; we note that $z\notin N_{12}$, but $z\in \widetilde N_{12}$.
    Let $\widetilde S\subset \widetilde N_{12}$ be a finite generating set, and define $S\subset N_{12}$ by
    \[S=\{sz^{-q(s)}\,|\,s\in \widetilde S\}.\]
    Then $S$ is a finite generating set for $N_{12}$; to see this, given $\gamma\in N_{12}$, write it as a finite product of elements of $\widetilde S$, and then introduce cancelling copies of $z$ and $z^{-1}$ to rewrite this product as a product of elements of $S$.
\end{proof}

\begin{proposition}\label{pr:N0nfg}
    The Specht subgroup $N_0$ is not finitely generated, because its abelianization is not finitely generated.
\end{proposition}

\begin{proof}
    There is a homomorphism $N_0\to \Z$ by $\gamma\mapsto \omega_{12}(\gamma)$.
    By Proposition~\ref{pr:windeqns}, the kernel of this map is $PB'_n$, since in $N_0$, $\omega_{ij}(\gamma)=\omega_{12}(\gamma)$ for all $\gamma$.
    This gives us an exact sequence
    \[1\to PB'_n\to N_0\to \Z\to 1.\]
    As above, we let $z\in PB_n$ denote our usual generator for the center of $PB_n$.
    It is important to note that $z\in N_0$ and $z\mapsto 1$ under the map $N_0\to \Z$ in this exact sequence.
    This means that the action of $\Z$ on $H_1(PB'_n)$ is trivial, since we can compute this action by lifting $1$ to $z$ and conjugating by $z$.
    From the five-term exact sequence for the low-dimensional homology of this exact sequence, we have
    \[H_2(\Z)\to H_1(PB'_n)_\Z\to H_1(N_0)\to H_1(\Z)\to 0.\]
    Since the action of $\Z$ on $H_1(PB'_n)$ is trivial, $H_1(PB'_n)_\Z$ is equal to $H_1(PB'_n)$.
    By Proposition~\ref{pr:H1commutator}, $H_1(PB'_n)$ is not finitely generated.
    Since $H_2(\Z)$ is trivial, it follows that $H_1(N_0)$ is not finitely generated.
\end{proof}

\subsection{Quotients by Specht subgroups}\label{se:ssquot}
The quotients of $PB_n$ by the Specht subgroups are interesting, since if $N$ is a Specht subgroup, we have that $PB_n/N$ is a $\Z\sg{n}$-module and $B_n/N$ is an extension of $\sg{n}$ by $PB_n/N$.
It turns out that the quotients of $PB_n$ by Specht subgroups are not isomorphic to the submodules of $M^2_\Z$ that we have already described, so we describe them here.
These descriptions play a role in our computation of cohomology below, so in this subsection, we work over a general commutative ring $R$.
Even though $R$ is a general commutative ring, we still use the notation $r\equiv s\pmod{p}$ to mean that $r-s\in pR$.

%It will be important later to understand the kernels of $\pi^1$ and $\pi^2$ working over $\Z$.
To get as much use as possible out of our computations when working over a general ring, we define maps that are the integral rescalings of $\pi^1$ and $\pi^2$ that just clear the denominators.
This works out differently depending on whether $n$ is even or odd, so we use some \emph{ad hoc} notations to work in a unified way.
For the remainder of the paper, we use
\[a=\left\{\begin{array}{cc}\frac{1}{2}&\text{$n$ even} \\ 1 &\text{$n$ odd}\end{array}\right.
\quad\text{ and }\quad
b=\left\{\begin{array}{cc}1 &\text{$n$ even} \\ \frac{1}{2} &\text{$n$ odd.}\end{array}\right.\]
To simplify notation, we compose $\pi^1$ with an isomorphism to $S^1_\Q\subset M^1_\Q$ before rescaling.
%We return to working over a general commutative ring $R$.

\begin{definition}\label{de:fimaps}
    For $n\geq 2$, the map $f^0\co M^2_R\to S^0_R$ is simply $f^0(v_{ij})=1$.

    For $n\geq 2$, the map $f^1\co M^2_R\to S^1_R$ is 
    \[f^1(v_{ij})=na(t_i+t_j)-2a \sum_{k=1}^n t_k.\]

    For $n\geq 4$, the map $f^2\co M^2_R\to S^2_R$ is 
    \[f^2(v_{ij})= (n-1)(n-2)bv_{ij}-(n-1)b(w_i+w_j)+2bu.\]
\end{definition}
(In these formulas, we divide by $2$ in the integers before scaling by the result in the $R$-module.)
We call out the key property of these maps, which we use repeatedly, but which needs no further proof.
\begin{lemma}\label{le:fiispii}
    Let $R=\Z$, and let $\nu\co M^1_\Z\to M^2_\Z$ be the map $\nu(t_i)=w_i$.
    Then $\nu\circ f^1=n(n-2)a\pi^1$ and $f^2=(n-1)(n-2)b\pi^2$.
    In particular, $\ker(f^i)=\ker(\pi^i|_{M^2_\Z})$ for $i=1,2$, and $\ker(f^1\circ\pi)=N_{02}$ and $\ker(f^2\circ \pi)=N_{01}$.
\end{lemma}

Now we compute the image of $f^1$.
\begin{proposition}\label{pr:f1image}
    An element $v=\sum_{i=1}^n c_it_i$ is in $f^1(M^2_R)$ if and only if $v\in S^1_R$ and for all $i$ and $j$, 
    \[c_i\equiv c_j\pmod{na}.\]
\end{proposition}

\begin{proof}
    The easy inclusion is that every $v\in f^1(M^2_R)$ satisfies these conclusions.
    The map $M^1_R\to R$ by $t_i\mapsto 1$ sends every $f^1(v_{ij})$ to $0$.
    This means that $f^1(v_{ij})\in S^1_R$.
    The coefficient on $t_i$ or $t_j$ of $f^1(v_{ij})$ is $na-2a$, and the coefficient on any $t_k$ is $-2a$ (with $k\notin\{i,j\}$), and these are all congruent modulo $na$.
    This finishes this direction of the argument.

    Now suppose $v=\sum_i c_it_i$ is in $S_R^1$ and satisfies the system of congruences in the statement.
    We modify $v$ by adding elements of $f^1(M^2_R)$ until we get $0$.
    If $n$ is even, then $f^1(v_{12})$ has $-1$ as its coefficient on $t_n$.
    If $n$ is odd, then 
    \[f^1(v_{23}+v_{45}+\dotsm+v_{n-1,n})=(1-n)t_i+t_2+t_3+\dotsm+t_n.\]
    So by adding an appropriate multiple of one of these elements, we can assume that $c_n$ is $0$.
    Now note that $f^1(v_{1n}-v_{1i})=t_n-t_i$.  
    By adding multiples of these elements, we can assume that $c_2$ through $c_{n-1}$ also equal $0$.
    Finally, since $v\in S^1_R$, we have $\sum_{i=1}^n c_i=0$, which implies that $v=0$.
    So the original $v$ was in $f^1(M^2_R)$.
\end{proof}

We note two corollaries about the Specht subgroups.
\begin{corollary}\label{co:f1M2index}
    We have that $PB_n/N_{02}$ embeds in $N_1/PB_n'$ as a $\Z\sg{n}$-submodule of index $(na)^{n-2}$.
\end{corollary}

\begin{proof}
    We use that 
    \[PB_n/N_{02}\cong M^2_\Z/\ker(f^1)\cong f^1(M^2_\Z),\]
    and $N_1/PB_n'\cong S^1_\Z$.
    Another interpretation of Proposition~\ref{pr:f1image} is that $f^1(M^2_\Z)$ is the kernel of the following map, which is the composition of reduction modulo $na$ and a quotient map:
    \[S^1_\Z\to S^1_{\Z/na\Z} \to S^1_{\Z/na\Z}/\Delta.\]
    Here $\Delta$ is the submodule spanned by $\sum_{i=1}^nt_i$.
    This map is clearly surjective.

    We need to take a moment to study $S^1_{\Z/na\Z}$.
    Since $\sum_{i=1}^nt_i$ maps to $0$ under the map $M^1_{\Z/na\Z}\to \Z/na\Z$, it is in $S^1_{\Z/na\Z}$.
    The sum of all the basis elements $t_2-t_1,\dotsc,t_n-t_1$ is 
    \[(1-n)t_1+t_2+t_3+\dotsm+t_n,\]
    which is congruent to $\sum_{i=1}^nt_i$.
    So $S^1_{\Z/na\Z}$ has a basis in which $\sum_{i=1}^nt_i$ is a member, and $S^1_{\Z/na\Z}/\Delta$ is a free $(\Z/na\Z)$-module of rank $n-2$.
\end{proof}

\begin{corollary}\label{co:f1M2isomorphism}
    Even though $PB_n/N_{02}$ and $N_1/PB_n'$ are isomorphic as abelian groups, and $(PB_n/N_{02})\otimes_\Z\Q$ and $(N_1/PB_n')\otimes_\Z\Q$ are isomorphic as $\Q\sg{n}$-modules, we have that $PB_n/N_{02}\not \cong N_1/PB_n'$ as $\Z \sg{n}$-modules.
\end{corollary}

\begin{proof}
    As just noted in Corollary~\ref{co:f1M2index}, $PB_n/N_{02}\cong f^1(M^2_\Z)$ and $N_1/PB_n'\cong S^1_\Z$.
    Both $S^1_\Z$ and $f^1(M^2_\Z)$ are free abelian groups of rank $n-1$, and since one is a finite-index submodule of the other, they become the same module after tensoring with $\Q$.
    However, the $-1$-eigenvectors of transpositions fit into these modules in different ways.
    In $S^1_\Z$, the $-1$-eigenvectors of $s_i$ are all the vectors of the form $c(t_i-t_{i+1})$, for $c\in \Z$.
    In $f^1(M^2_\Z)$, the $-1$-eigenvectors of $s_i$ are all the vectors of the form $cna(t_i-t_{i+1})$, for $c\in \Z$.
    In particular, $-1$-eigenvectors of transpositions generate $S^1_\Z$, but $-1$-eigenvectors of transpositions in $f^1(M^2_\Z)$ generate $naS^1_\Z$, which is a proper submodule.
    Since being generated by $-1$-eigenvectors of transpositions is a property that would be preserved by a $\Z\sg{n}$-module isomorphism, these modules are not isomorphic.
\end{proof}

Next we do something similar for $f^2$. 
Our immediate goal is to say something interesting about $PB_n/N_{01}$, but the computations here will also be useful later when we study splitting of extensions.
\begin{proposition}\label{pr:f2M2congruences}
    Let $v=\sum_{ij}c_{ij}v_{ij}\in M^2_R$.
    Then $v\in f^2(M^2_R)$ if and only if $v$ satisfies the following four conditions:
    \begin{itemize}
        \item $v\in S^2_R$;
        \item for all unordered pairs $\{i,j\}$ and $\{k,l\}$, possibly not disjoint, taken from $\{1,2,\dotsc,n\}$, we have
        \begin{equation}\label{eq:coeffcongruence}
            c_{ij}\equiv c_{kl}\pmod{(n-1)b};
        \end{equation}
        \item for any distinct $i,j,k,l$ from $\{1,2,\dotsc,n\}$ we have
        \begin{equation}\label{eq:polytabloidcongruence}
        c_{ij}+c_{kl}\equiv c_{il}+c_{kj}  \pmod{n-2}; \text{ and}
        \end{equation}
        \item if $n$ is even, then for any distinct $i,j,k$ from $\{1,2,\dotsc,n\}$, we have
        \begin{equation}\label{eq:paritybit}
        c_{ij}+c_{ik}+c_{jk}\equiv 0 \pmod{2}.
        \end{equation}
    \end{itemize}
\end{proposition}

\begin{proof}
    Before we run the double-inclusion argument, we note the coefficients of $f^2(v_{ij})$.
    Specifically, if the $c_{kl}\in R$ are the coefficients such that $\sum_{kl}c_{kl}v_{kl}=f^2(v_{ij})$, then
    \[
c_{kl}=
\left\{
\begin{array}{cl}
 (n^2-5n+6)b & \text{ if $\{i,j\}=\{k,l\}$,}\\
 (3-n)b & \text{if $\abs{\{i,j\}\cap\{k,l\}}=1$} \\
2b & \text{if $\{i,j\}\cap\{k,l\}=\varnothing$.}
\end{array}\right.
\]
    This follows from Definition~\ref{de:fimaps} by direct computation.

    Using this, we show the first direction: if $v\in f^2(M^2_R)$, then it satisfies the conditions above.
    We show this by showing it for generators; so let $v=f^2(v_{ij}).$
    To show that $v\in S^2_R$, we use Proposition~\ref{pr:SESS2K12S1}.
    First map $f^2(v_{ij})$ to $S^0_R$ by the map that sends every standard generator to $1$.
    Note that $f^2(v_{ij})$ has one coefficient of $(n^2-5n+6)b$, $2(n-2)$ coefficients of $(3-n)b$, and $\binom{n-2}{2}$ coefficients of $2b$.
    From this it follows that the sum of the coefficients is $0$, so that $f^2(v_{ij})$ is in $K_{12}$ (as defined in Proposition~\ref{pr:SESS2K12S1}).
    Next we evaluate $\mu(f^2(v_{ij}))$; recall that $\mu\co M^2_R\to M^1_R$ is the map with $v_{kl}\mapsto t_k+t_l$.
    The coefficient on $t_k$ in $\mu(f^2(v_{ij}))$ is 
    \[
    \left\{
    \begin{array}{cl}
    (1(n^2-5n+6)+(n-2)(3-n))b & \text{ $k\in\{i,j\}$,}\\
    (2(3-n)+(n-3)\cdot 2)b & \text{otherwise,}
    \end{array}\right.
    \]
    so $\mu(f^2(v_{ij}))=0$.
    Then by Proposition~\ref{pr:SESS2K12S1}, $f^2(v_{ij})\in S^2_R$.
    Since elements of this form span $f^2(M^2_R)$, we have $f^2(M^2_R)\subseteq S^2_R$.

    Next we check that $f^2(v_{ij})$ satisfies system~\eqref{eq:coeffcongruence}.
    Note that $(n^2-5n+6)b=((n-1)(n-4)+2)b$, and $(3-n)b=(-(n-1)+2)b$.
    This implies that all coefficients of $f^2(v_{ij})$ are congruent to $2b$ modulo $(n-1)b$.
    In particular, they are all congruent to each other.
    This is a homogeneous system of congruences, so its solution set is a submodule, and so the entirety of $f^2(M^2_R)$ satisfies it.

    Now we check that $f^2(v_{ij})$ satisfies system~\eqref{eq:polytabloidcongruence}.
    Suppose we want to verify that $c_{kl}+c_{rs}\equiv c_{ks}+c_{rl}\pmod{n-2}$.
    Up to rearranging the congruence, there are four cases: (1) that $\{i,j\}=\{k,l\}$;
     (2) that each of $\{k,l\}$,$\{r,s\}$, $\{k,s\}$ and $\{r,l\}$ intersect $\{i,j\}$ in a single element; (3) $\{i,j\}\cap\{k,l\}$ have a single element and $\{i,j\}$ and $\{r,s\}$ are disjoint; or 
    (4) that $\{i,j\}$ and $\{k,l,r,s\}$ are disjoint.
    The congruence in the first case is equivalent to
    \[((n-2)(n-3)+2)b\equiv 2(1+(2-n))b\pmod{n-2}.\]
    In the second case it is
    \[2(3-n))b\equiv 2(3-n))b\pmod{n-2}.\]
    In the third case it is
    \[((3-n)+2)b\equiv ((3-n)+2)b\pmod{n-2},\]
    and in the fourth case it is $4b\equiv 4b\pmod{n-2}$.
    So the congruence holds in all cases; and since the indices were arbitrary, the entire system holds for $f^2(v_{ij})$.
    Again, this is a homogeneous system of congruences and the solution set is a submodule, so all vectors in $f^2(M^2_R)$ satisfy it.

    Finally, we check that $f^2(v_{ij})$ satisfies system~\eqref{eq:paritybit}.
    We assume that $n$ is even, since there is no condition here if $n$ is odd.
    For three distinct $k,l,m\in\{1,2,\dotsc,n\}$, there are three cases to consider: that $i,j\in \{k,l,m\}$, that $\{k,l,m\}$ contains exactly one of $i$ or $j$, or that $\{k,l,m\}$ is disjoint from $\{i,j\}$.
    In the first case, $c_{kl}+c_{km}+c_{lm}$ is $n^2-5n+6+2(3-n)$, which is even.
    In the second case, $c_{kl}+c_{km}+c_{lm}$ is $2(3-n)+2$, which is even.
    In the third case, $c_{kl}+c_{km}+c_{lm}=6$.
    So the system of congruences is true for $f^2(v_{ij})$, and therefore for all of $f^2(M^2_R)$.
    This finishes the forward inclusion of the argument.

    Now suppose that $v=\sum_{ij}c_{ij}v_{ij}\in M^2_R$ satisfies the four conditions in the statement.
    We alter $v$ by adding elements of $f^2(M^2_R)$ until we get $0$, to show that the original $v$ was in $f^2(M^2_R)$.
    We do this in several steps.

    Step 1: Zero out $c_{n-1,n}$.
    If $n$ is odd, then the $v_{n-1,1}$-coefficient of $f^2(v_{12})$ is $1$.
    If $n$ is even, then the $v_{n-1,1}$-coefficient of $f^2(v_{12})$ is $2$, and $v_{n-1,1}$-coefficient of $f^2(v_{1n})$ is $3-n$, which is odd.
    So by adding some linear combination of $f^2(v_{12})$ and $f^2(v_{1n})$ to $v$, we may assume that $c_{n-1,n}=0$.

    Step 2: Zero out the other $c_{i,n}$.
    By the system of congruences~\eqref{eq:coeffcongruence}, since $c_{n-1,n}=0$, we have that $b(n-1)$ divides every coefficient of $v$.
    For each $i>2$, the element $f^2(v_{12}-v_{1i})$ has $0$ for its coefficients on $v_{i+1,n}$ through $v_{n-1,n}$, but has $b(n-1)$ as the $v_{i,n}$-coefficient.
    So we can use these to zero out $c_{n-2,n}$ through $c_{3n}$.
    The element $f^2(v_{13}-v_{23})$ has $0$ as its coefficient on $v_{3n}$ through $v_{n-1,n}$, but has $b(n-1)$ as its $v_{2n}$-coefficient.  
    So we can use this to zero out $c_{2n}$.
    This then implies that $c_{1n}=0$: since $v\in S^2_M$, $v$ is a linear combination of standard polytabloids, but any polytabloid that has $0$ as its coefficients on $v_{2n}$ through $v_{n-1,n}$ also has $0$ as its coefficient on $v_{1n}$.

    Step 3: Zero out $c_{n-2,n-1}$.
    We break into cases depending on whether $n$ is even or odd.
    
    Case: $n$ odd.
    Let $P$ be a partition of $\{1,2,\ldots,n-1\}$ where all cells have size $2$, and such that $\{n-2,n-1\}$ is not in $P$.
    Define $v_P=\sum_{\{i,j\}\in P}v_{ij}$.
    For each $i$ from $1$ to $n-1$, the $v_{i,n}$-coefficient of $f^2(v_P)$ is $(3-n)/2+((n-1)/2-1)\cdot 1=0$ 
    (this uses $b=1/2$, and one summand contributes $(3-n)/2$ since $i$ is in exactly one cell in $P$).
    The coefficient on $v_{n-2,n-1}$ is $2((3-n)/2)+((n-1)/2-2)\cdot 1=(n-1)/2$ (two summands contribute $(3-n)/2$ since $n-2$ and $n-1$ are each in distinct cells in $P$).
    As in the previous step, we know that $(n-1)/2$ divides $c_{n-2,n-1}$.
    So adding some multiple of $f^2(v_P)$ will zero out $c_{n-2,n-1}$.

    Case: $n$ even.
    In this case, we need to use that $c_{n-2,n-1}$ is in $2R$.
To show this, we use the system of congruences~\eqref{eq:paritybit}.
Choosing our triple of indices to be $n-2$, $n-1$, and $n$, we have that $c_{n-2,n-1}+c_{n-2,n}+c_{n-1,n}$ is in $2R$.
But $c_{n-2,n}=c_{n-1,n}=0$, so $c_{n-2,n-1}\in 2R$.
Now we proceed with this case.
Let $P$ be a set of $n-1$ subsets of $\{1,2,\ldots, n-1\}$ of size $2$, such that each $i$ appears in exactly two subsets in $P$.
In other words, $P$ is the set of unordered pairs corresponding to the edges of a graph whose vertices are $\{1,2,\ldots, n-1\}$ and where every vertex has degree $2$ (so the graph is a disjoint union of cycles).
Also assume that $\{n-2,n-1\}$ is not in $P$.
Again define $v_P=\sum_{\{i,j\}\in P}v_{ij}$.
For each $i$ from $1$ to $n-1$, the $v_{i,n}$-coefficient of $f^2(v_P)$ is $2(3-n)+(n-3)\cdot 2=0$, since $i$ appears in exactly two sets in $P$.
The coefficient on $v_{n-2,n-1}$ is $4(3-n)+(n-5)\cdot 2=-2(n-1)$, since $\{n-2,n-1\}$ overlaps $4$ sets in $P$ in a single element, and is disjoint from the remaining $n-5$ sets.
So let $r\in R$ with $c_{n-2,n-1}=2r$, and replace $v$ with $v-f^2(rv_P)$.
In the new $v$, we have zeroed out $c_{n-2,n-1}$ without disturbing the coefficients on any $v_{i,n}$.

Step 4: Recognize $v$ as a linear combination of polytabloids in $f^2(M^2_R)$.
Fix $i$ and $j$ with $c_{ij}$ not already zeroed out.
By system~\eqref{eq:coeffcongruence}, $c_{ij}\in (n-1)bR$.
From system~\eqref{eq:polytabloidcongruence}, we know
\[c_{i,n-2}+c_{n-1,n}\equiv c_{i,n}+c_{n-1,n-2}\pmod{n-2},\]
meaning that $c_{i,n-2}\in (n-2)R$.
The same argument works for $j$, so $c_{j,n-2}\in (n-2)R$.
We also have
\[c_{ij}+c_{n-1,n}\equiv c_{i,n}+c_{j,n-1}\pmod{n-2},\]
so that $c_{ij}\in (n-2)R$.
But $(n-1)b$ and $(n-2)$ are relatively prime in $\Z$, and therefore $c_{ij}\in (n-1)(n-2)bR.$
So $v\in (n-1)(n-2)bS^2_R$.
By Lemma~\ref{le:fiispii}, we know that for any standard polytabloid $e_{ij}$, we have $f^2(e_{ij})=(n-1)(n-2)be_{ij}$.
So find a vector $v'\in S^2_R$ with $(n-1)(n-2)bv'=v$, write $v'$ as a linear combination of standard polytabloids, and then recognize $v=f^2(v')$.
\end{proof}

\begin{lemma}\label{le:removefourthcondition}
    If $R$ has trivial $\frac{n-2}{2}$-torsion (that is, $R[\frac{n-2}{2}]=0$), then system~\eqref{eq:paritybit} in Proposition~\ref{pr:f2M2congruences} is unnecessary.
\end{lemma}

\begin{proof}
    To prove this, we explain how to modify the proof of the second direction of Proposition~\ref{pr:f2M2congruences} with the modified hypotheses.

    So suppose $v=\sum_{ij}c_{ij}v_{ij}\in S^2_M$ and it satisfies systems~\eqref{eq:coeffcongruence} and~\eqref{eq:polytabloidcongruence}, but not system~\eqref{eq:paritybit}, \emph{a priori}.
    We carry out the steps in the reduction, and get to step 3, which is to show that we can zero out $c_{n-2,n-1}$.
    The case where $n$ is odd does not use the system~\eqref{eq:paritybit}, so we work the case where $n$ is even.
    We need to show that $c_{n-2,n-1}\in 2R$, since if it is, we have an element of $f^2(M^2_R)$ that we can use to zero it out.
    Since we have worked the first two steps, we have that $c_{1,n},\dotsc,c_{n-1,n}$ are all $0$.
    By system~\eqref{eq:polytabloidcongruence},
    \[c_{i,n-1}+c_{n-2,n}\equiv c_{n-1,n-2}+c_{i,n}\pmod{n-2},\]
    so $c_{1,n-1},\dotsc,c_{n-1,n-2}$ are all congruent modulo $n-2$.
    Also, for any $i$ and $j$ with $1\leq i<j\leq n-1$,
    \[c_{ij}+c_{n-2,n}\equiv c_{j,n-2}+c_{i,n} \pmod{n-2}.\]
    This shows that all $\binom{n-1}{2}$ coefficients $c_{12},\dotsc,c_{n-2,n-1}$ are congruent modulo $n-2$.
    But $v\in S^2_R$, so $\sum_{ij}c_{ij}=0$.
    So
    \[(n-1)\frac{n-2}{2}c_{n-1,n-2}\equiv 0 \pmod{n-2}\]
    Since $n-1\equiv 1\pmod{n-2}$, we can ignore it, and we have $r\in R$ with $(n-2)r=\frac{n-2}{2}c_{n-1,n-2}$.
    Now we use the hypothesis that $R$ has trivial $\frac{n-2}{2}$-torsion.
    We have that $\frac{n-2}{2}(c_{n-1,n-2}-2r)=0$, but our hypothesis then implies that $c_{n-1,n-2}-2r=0$, in other words, that $c_{n-1,n-2}$ is a multiple of $2$ in $R$.
    Steps 3 and 4 can then proceed as above.
\end{proof}

Our next goal is to prove corollaries like Corollary~\ref{co:f1M2index} and Corollary~\ref{co:f1M2isomorphism} for $f^2(M^2_\Z)$.
As in the proof of Corollary~\ref{co:f1M2index}, we want to describe $f^2(M^2_\Z)$ as the kernel of a map from $S^2_\Z$ to a finite module; however, in this case the finite module is more difficult to work with.

\begin{definition}\label{de:barepsilon}
    Let $S^{2*}_R$ be the $R\sg{n}$-module given by $\hom_R(S^2_R,R)$, with the action $(\sigma f)(v)=f(\sigma^{-1}v)$.
    For $v, w\in M^2_R$, let $v\cdot w\in R$ be the inner product with respect to the standard basis for $M^2_R$.
    Let $\epsilon\co M^2_R\to S^{2*}_R$ be the map $v\mapsto \epsilon_v$, where $\epsilon_v(w)=v\cdot w$, for $v\in M^2_R$ and $w\in S^2_R$.
    Let $\bar\epsilon \co M^2_R\to S^{2*}_{R/(n-2)R}$ be the composition
     \[M^2_R\stackrel{\epsilon}{\longrightarrow} S^{2*}_R \to S^{2*}_{R/(n-2)R},\]
    where the unlabeled map is reduction modulo $n-2$.   
\end{definition}

\begin{lemma}\label{le:kerepsilon}
    For any $R$, the map $\epsilon\co M^2_R\to S^{2*}_R$ is an $R\sg{n}$-module homomorphism.
    %Recall the second system of congruences from Proposition~\ref{pr:f2M2congruences}: for any distinct $i,j,k,l$ from $\{1,2,\dotsc,n\}$,    
    %\[c_{ij}+c_{kl}\equiv c_{il}+c_{kj}  \pmod{n-2}.\] 
    The submodule of $S^2_R$ consisting of vectors satisfying system~\eqref{eq:polytabloidcongruence} is equal to $\ker(\bar\epsilon|_{S^2_R})$.
\end{lemma}

\begin{proof}
    It is easy to see that the inner product on $M^2_R$ is $\sg{n}$-invariant, meaning that for any $\sigma\in \sg{n}$ and any $v,w\in M^2_R$, we have $v\cdot w=(\sigma v)\cdot(\sigma w)$.
    Then for $\sigma\in \sg{n}$, $v\in M^2_R$ and $w\in S^2_R$, we have that 
    \[\epsilon_{\sigma v}(w)=(\sigma v)\cdot w=v\cdot (\sigma^{-1}w)=\epsilon_v(\sigma^{-1}w)=(\sigma \epsilon_v)(w).\]
    Since $R$-linearity of $\epsilon$ follows from bilinearity of the inner product, this means that $\epsilon$ is an $R\sg{n}$-module homomorphism.

    For the second statement, let $v=\sum_{ij}c_{ij}v_{ij}$ and let $w=v_{ij}+v_{kl}-v_{il}-v_{kj}$ be a polytabloid.
    Note that $v$ satisfies system~\eqref{eq:polytabloidcongruence} if and only if $v\cdot w\in (n-2)R$, if and only if $\epsilon_v(w)\in (n-2)R$.
    So $v$ satisfies the full system of congruences if and only if $\epsilon_v(w)\in (n-2)R$ for all polytabloids $w$.
    So if $v$ is in the kernel of $\bar\epsilon$, then $v$ satisfies the system of congruences.
    Conversely, if $v$ satisfies the system of congruences, then $v$ is in the kernel of $\bar{\epsilon}$ because polytabloids generate $S^2_R$.
\end{proof}

\begin{proposition}\label{pr:S2*presentation}
    The $(R/(n-2)R)\sg{n}$-module $\bar \epsilon(S^2_R)\subseteq S^{2*}_{R/(n-2)R}$ has the following presentation as an $(R/(n-2)R)$-module.
    \begin{itemize}
        \item Generators are $\epsilon(e_{ij})$ for $2\leq i<j\leq n-1$ and $(i,j)\neq(2,3)$, the set of images of standard polytabloids not involving $n$.
        \item The relations are linear combinations of generators of the form $c v_0=0$, where $v_0=\sum_{ij}c_{ij}\epsilon(e_{ij})$, with 
        \[
        c_{ij}=
        \left\{
        \begin{array}{cl}
        1 & \text{if $i\geq 3$ or $(i,j)=(2,n-1)$,}\\
        n-j & \text{ if $i=2$,}
        \end{array}
        \right.
        \]
        and $c\in R/(n-2)R$ with $\binom{n-2}{2}c=0$.
    \end{itemize}
\end{proposition}

\begin{remark}
    If $n$ is odd, then $n-2$ divides $\binom{n-2}{2}$ in $\Z$, and the condition on $c$ in the presentation is vacuous.
    So if $n$ is odd, take $c=1$, and the relation simply eliminates one of the generators, say $\bar\epsilon(e_{n-2,n-1})$.

    If $n$ is even, then $n-2$ does not divide $\binom{n-2}{2}$ in $\Z$, and the condition is meaningful.
    In this case, $\bar\epsilon(v_0)$ is may be nontrivial (depending on $R$), but $\bar\epsilon(2v_0)$ is definitely trivial.
\end{remark}

\begin{proof}[Proof of Proposition~\ref{pr:S2*presentation}]
    We know that $S^2_R$ is generated by standard polytabloids, so to prove the statement about generators, we need to show that for every standard polytabloid $e_{i,n}$, we can express $\epsilon(e_{i,n})$ in terms of the generators.
    For $i=2$, one can compute that
    \[\sum_{j=4}^n e_{2j}=(n-3)v_{13} -(n-3)v_{23} +\sum_{j=4}^n(v_{2j}-v_{1j}).\]
    So $\sum_{j=4}^ne_{2j}\equiv w_1-w_2\pmod{n-2}$.
    (Recall that $w_i$ is the sum of all standard basis elements with $i$ as an index.)
    However, $w_k\cdot e_{ij}=0$ for any $i,j,k$, even if $k\in \{i,j\}$.
    So 
    \[\bar\epsilon(e_{2n})=\sum_{i=4}^{n-1}\bar\epsilon(e_{2i}).\]
    This equation remains true if we act on both sides by a permutation, however, acting on a standard polytabloid by a permutation may give us a polytabloid that is not standard.
    So to prove a similar fact for $i\neq 2$, act on both sides by the transposition $(2i)\in \sg{n}$, and then rewrite the nonstandard polytabloids in terms of standard polytabloids.
    This still works for showing that our generating set is correct, since a nonstandard polytabloid not involving $n$ as an index can be rewritten as a linear combination of standard polytabloids not involving $n$.

    Now we prove the statement about relations.
    Suppose that $v=\sum c_{ij}e_{ij}$, where the sum is over standard polytabloids not involving $n$, and suppose $\bar\epsilon(v)=0$.
    Like in Step 3 of the proof of Proposition~\ref{pr:f2M2congruences}, the fact that $c_{i,n}=0$ for all $i$ implies that $c_{2,n-1}\equiv c_{i,n-1}$ for all $i$, and that $c_{ij}\equiv c_{i,n-1}$ for all $i$ and $j$ (by dotting with $v_{2,n-1}+v_{i,n}-v_{2,n}-v_{i,n-1}$ and $v_{ij}+v_{n-1,n}-v_{i,n-1}-v_{nj}$ respectively).
    So for all $i,j$ with $i\geq 3$, we have $c_{ij}=c_{2,n-1}$.

    Next, consider $w=v_{2j}+v_{n-1,n}-v_{2,n-1}-v_{nj}$.
    The standard polytabloids not involving $n$ and with a nontrivial coefficient on $v_{2j}$ are $e_{2j}$, $e_{j,j+1}$, $e_{j,j+2}$, $\ldots$, $e_{j,n-1}$.
    So 
    \[0=v\cdot w=c_{2,j}-(\sum_{k=1}^{i-1}c_{j,n-k})-c_{2,n-1}.\]
    Since most of these are equal, we see that $c_{2j}=ic_{2,n-1}$.
    This implies that $v=c_{2,n-1}v_0$, where $v_0$ is the vector in the statement.

    Now set $w=v_{23}+v_{n-1,n}-v_{2,n-1}-v_{3,n}$.
    The standard polytabloids not involving $n$ and with a nontrivial coefficient in $v_{23}$ are $e_{2,4}$ through $e_{2,n-1}$ and $e_{3,4}$ through $e_{3,n-1}$.
    So
    \[0\equiv -v\cdot w=((\sum_{i=1}^{n-4}i)+(n-4)+1)c_{2,n-1}=\frac{(n-2)(n-3)}{2}c_{2,n-1}.\]
    So the class of $c_{2,n-1}$ is in the $\binom{n-2}{2}$-torsion of $R/(n-2)R$.
\end{proof}

Now we can prove our corollaries about Specht subgroups.
\begin{corollary}\label{co:f2M2index}
The $\Z \sg{n}$-module $PB_n/N_{01}$ embeds in $N_2/PB_n'$ as a submodule of index 
\[2b\cdot((n-1)b)^{\binom{n}{2}-n-1}\cdot (n-2)^{\binom{n-1}{2}-n}.\]
\end{corollary}

\begin{proof}
    We recognize $PB_n/N_{01}\cong f^2(M^2_\Z)$ and $N_2/PB_n'\cong S^2_\Z$.
    By reduction modulo $(n-1)b$ and taking the quotient, we have a homomorphism $S^2_\Z\to S^2_{\Z/(n-1)b\Z}/\Delta$, where $\Delta$ is the submodule spanned by $u$ (the sum of all the standard basis elements).
    This homomorphism is clearly surjective.
    We also have the homomorphism $\bar\epsilon\co S^2_\Z\to S^{2*}_{\Z/(n-2)\Z}$.
    We combine these into a homomorphism
    \[S^2_\Z\to S^2_{\Z/(n-1)b\Z}/\Delta\times \bar\epsilon(S^2).\]
    The kernel of this product homomorphism is the intersection of the kernels of the factors.
    We recognize the kernel of the first factor as being the submodule of $S^2_\Z$ satisfying system~\eqref{eq:coeffcongruence}, and by Lemma~\ref{le:kerepsilon}, we recognize the kernel of the second factor as being the submodule satisfying the system~\eqref{eq:polytabloidcongruence}.
    By Lemma~\ref{le:removefourthcondition}, the intersection of these two kernels is exactly $f^2(M^2_\Z)$ (since $\Z$ has no torsion).
    So $f^2(M^2_\Z)$ is the kernel of a homomorphism from $S^2_\Z$ to a finite module.
    To finish the statement, we need to compute the order of the module and deduce that the map is surjective.

    Note that under the map to $\Z/(n-1)b\Z$,
    we have $u\mapsto \binom{n}{2}=0$.
   % , which is $0$ in $\Z/(n-1)b\Z$.
    We have $\mu(u)=(n-1)\sum_kt_k$ ($\mu$ as in Proposition~\ref{pr:SESS2K12S1}), so by Proposition~\ref{pr:SESS2K12S1}, $u$ is in $S^2_{\Z/(n-1)b\Z}$.
    Notice that $u-e_{n-1,n}$ has $0$ as its coefficient on $v_{n-1,n}$, and can therefore be written as a linear combination of standard polytabloids not involving $n$.
    This means that we can replace $e_{n-1,n}$ by $u$ in our basis for $S^2_{\Z/(n-1)b\Z}$ and get another basis.
    This implies that $S^2_{\Z/(n-1)b\Z}/\Delta$ is a free $(\Z/(n-1)b\Z)$-module of rank $\binom{n}{2}-n-1$.
    
    Proposition~\ref{pr:S2*presentation} says that $\bar\epsilon(S^2_\Z)$ is a direct sum of a free $(\Z/(n-2)\Z)$-module of rank $\binom{n-1}{2}-n$, and a factor that is $\Z/2\Z$ if $n$ is even, and trivial if $n$ is odd.
    So the order of $\bar\epsilon(S^2_\Z)$ is $2b\cdot (n-2)^{\binom{n-1}{2}-n}$.
    %\[
    %\abs{\bar\epsilon(S^2_\Z)}=
    %\left\{
    %\begin{array}{cl}
    %(n-2)^{\binom{n-1}{2}-n} & \text{if $n$ is odd} \\
    %2(n-2)^{\binom{n-1}{2}-n} & \text{if $n$ is even.}
    %\end{array}
    %\right.
    %\]
    Whether $n$ is even or odd, the order of $S^2_{\Z/(n-1)b\Z}/\Delta$ is relatively prime to the order of $\bar\epsilon(S^2_\Z)$, and therefore the product homomorphism is surjective.
    The statement follows.
\end{proof}

\begin{corollary}\label{co:f2M2isomorphism}
    Even though $PB_n/N_{01}$ and $N_2/PB_n'$ are isomorphic as abelian groups, and $(PB_n/N_{01})\otimes_\Z\Q$ and $(N_2/PB_n')\otimes_\Z\Q$ are isomorphic as $\Q\sg{n}$-modules, we have that $PB_n/N_{01}\not \cong N_2/PB_n'$ as $\Z \sg{n}$-modules.
\end{corollary}

\begin{proof}
    Again, these are isomorphic as abelian groups because they are free abelian of the same rank, and they tensor to isomorphic modules over $\Q$ because one is finite index in the other.
    They are not isomorphic because we can characterize the polytabloids in terms of eigenvectors of transpositions.

    Claim: if $(ij),(kl)\in \sg{n}$ are disjoint transpositions, and $v\in M^2_\Z$ is a $-1$-eigenvector for both $(ij)$ and $(kl)$, then $v$ is a multiple of the polytabloid $v_{ik}+v_{jl}-v_{il}-v_{jk}$.

    The point is that the $-1$-eigenspace for $(ij)$ is spanned by 
    \[v_{1i}-v_{1j}, v_{2i}-v_{2j},\dotsc, v_{in}-v_{jn},\]
    (the other subscript ranges over the integers in $\{1,2,\dotsc,n\}$ other than $i$ and $j$).
    Something similar is true for $(kl)$, and the given polytabloid spans the intersection.
    This proves the claim.

    In $S^2_\Z$, the common $-1$-eigenvectors for disjoint transpositions generate, since $S^2_\Z$ is generated by polytabloids (by definition).
    Proposition~\ref{pr:f2M2congruences} implies that if $v$ is a polytabloid and $cv\in f^2(M^2_\Z)$, then $(n-1)(n-2)b$ divides $c$.
    So the common $-1$-eigenvectors of disjoint transpositions generate $(n-1)(n-b)bS^2_\Z$, a proper submodule.
    So $S^2_\Z$ and $f^2(M^2_\Z)$ cannot be isomorphic as $\Z\sg{n}$-modules.
\end{proof}

\subsection{The case when $n=3$}
In the preceding part of this section, we have a standing assumption that $n\geq 4$.
Many of our results are still true if $n=3$, but the overall setup is different, so we state the results carefully here.
First of all, for $n=3$, the definition of $M^2_R$ in terms of tabloids doesn't really work.
However, we can still define $M^2_R$ as the free $R$-module on $\{v_{12},v_{13},v_{23}\}$, with the obvious action by $\sg{3}$.
Then 
\[M^2_\Q\cong M^1_\Q\cong S^0_\Q\oplus S^1_\Q,\]
which is different from the deccomposition when $n\geq 4$.
The definition of $S^2_R$ specializes to the trivial submodule when $n=3$, and the formulas we give for $\pi^2$ and $f^2$ define the zero map.

Let $n=3$.
Let $i\in \{0,1\}$ and let $j\in \{0,1\}$ be the other element.
Define $N_i<PB_3$ to be the kernel of $\pi^j|_{M^2_\Z}\circ \pi$.
Then $N_0$ and $N_1$ are the \emph{Specht subgroups} of $PB_3$.
We note that Propositions~\ref{pr:normaliffsubmodule},~\ref{pr:N0nfg} and~\ref{pr:f1image} are still true if $n=3$, with no changes.
For the other results, some changes to the statements are needed; we collect all the modified statements in the following theorem.
After appropriate changes to the statements are made, no changes are needed to the arguments, so we give no further proof.
\begin{theorem}
    The following results are true when $n=3$.
    \begin{itemize}
        \item A version of Lemma~\ref{le:tensorwithQ}: For $i\in\{0,1\}$, $(N_i/PB_3')\otimes_\Z\Q\cong S^i_\Q$, and if $j\in \{0,1\}$ is the other element, $(PB_3/N_i)\otimes_\Z\Q\cong S^j_\Q$.
        \item A version of Lemma~\ref{le:pairwiseincomm}: The subgroups $PB_3'$, $N_0$, $N_1$, and $PB_3$ are pairwise incommensurable.
        \item A version of Theorem~\ref{th:classification}: Any normal subgroup $N\lhd B_3$ with $PB_3'\leq N\leq PB_3$ is a finite-index subgroup of exactly one of $PB_3'$, $N_0$, $N_1$, or $PB_3$.
        \item A version of Proposition~\ref{pr:windeqns}:
        \begin{itemize}
            \item For $\gamma\in PB_3$, $\gamma\in N_0$ if and only if 
            \[\omega_{12}(\gamma)=\omega_{13}(\gamma)=\omega_{23}(\gamma).\]
            \item For $\gamma\in PB_3$, $\gamma\in N_1$ if and only if
            \[\omega_{12}(\gamma)+\omega_{13}(\gamma)+\omega_{23}(\gamma)=0.\]
        \end{itemize}
        \item A version of Proposition~\ref{pr:biggensets}:
        \begin{itemize}
            \item $N_0$ is generated by the union of $PB_3'$ with the single element $a_{12}a_{13}a_{23}$.
            \item $N_1$ is generated by the union of $PB_3'$ with 
            \[\{a_{23}a_{13}^{-1},a_{23}a_{12}^{-1}\}.\]
        \end{itemize}
        \item A version of Proposition~\ref{pr:N12fg}: $N_1$ is finitely generated.
        \item A version of Corollary~\ref{co:f1M2index}:
        $PB_3/N_0$ is isomorphic to a $\Z \sg{3}$-submodule of $N_1/PB_3'$ with index $3$.
        \item A version of Corollary~\ref{co:f1M2isomorphism}:
        $PB_3/N_0$ and $N_1/PB_3'$ are isomorphic as abelian groups, and $(PB_3/N_0)\otimes_\Z\Q$ and $(N_1/PB_3')\otimes_\Z\Q$ are isomorphic as $\Q\sg{3}$-modules, but $PB_3/N_0\not\cong N_1/PB_3'$ as $\Z\sg{3}$-modules.
    \end{itemize}
\end{theorem}

\section{Preliminaries for cohomology}

\subsection{Low-dimensional cohomology in permutation modules}

In computing the cohomology of a group $G$, any projective resolution of $\Z$ as a $\Z G$-module can be used to build a cochain complex.
In the study of group extensions, the bar resolution is usually used, since bar $2$-cochains can be used directly to construct group extensions.
However, describing bar $2$-cochains can be difficult, since the domain is $G\times G$ and there is usually no good way to define a $2$-cochain on generators.

For this paper, we use the $3$-dimensional truncation of a free resolution coming from the presentation for $\sg{n}$ given above.
This is the same chain complex from Day-Nakamura~\cite{DN3skeleton}.
In that paper, this chain complex is given as the cellular chain complex of the universal cover of a $3$-dimensional cell complex $X$ with $\pi_1(X)=\sg{n}$ and with the universal cover $\widetilde X$ being $2$-connected.  
In this paper, we directly describe $P_*=C_*(\widetilde X)$ instead of building $X$.
The deck transformations of $\widetilde X\to X$ make $P_*$ into a chain complex of $\Z \sg{n}$-modules.

\begin{definition}[The complex $P_*$]\label{de:Snres}
Let $P_0$ be the free $\Z \sg{n}$-module on a single generator $*$.
Let $P_1$ be the free $\Z \sg{n}$-module on $n-1$ generators $e_1,\dotsc, e_{n-1}$.
For all $i$, $\partial e_i=(s_i-1)*$; this defines the chain boundary map $\partial\co P_1\to P_0$.
Now $P_2$ will be a free $\Z \sg{n}$-module on the following set of three classes of free generators.
We define $\partial\co P_2\to P_1$ simultaneously by giving the boundaries of the generators.
\begin{itemize}
    \item For each $i$ from $1$ to $n-1$, we have a generator $c_i$ for $P_2$, with $\partial c_i=(s_i+1)e_i$.
    \item For each $i$ from $1$ to $n-2$, we have a generator $b_i$ for $P_2$, with 
    \[\partial b_i= (1-s_i+s_{i+1}s_i)e_{i+1}-(1-s_{i+1}+s_is_{i+1})e_i.\]
    \item For each $i$ and $j$ with $1\leq i<j\leq n-1$ and $j\geq i+2$, we have a generator $d_{ij}$ for $P_2$, with $\partial d_{ij}=(s_j-1)e_i-(s_i-1)e_j.$
\end{itemize}
The specific $3$-cells in $P_*$ are not needed in the present paper, so their definitions are suppressed.
\end{definition}

\begin{remark}
The cells in this complex have geometric interpretations; for details and illustrations, see Remark~1.4 and Figures~1, 2, and~3 in Day--Nakamura~\cite{DN3skeleton}
\end{remark}

The following is Theorem~1.1 from Day--Nakamura~\cite{DN3skeleton}.
\begin{theorem}
    The complex $P_*$ is the $3$-dimensional truncation of a free resolution for $\Z$ as a $\Z \sg{n}$-module.
    %In particular, $H_0(P_*)=\Z$, and $H_1(P_*)=H_2(P_*)=0$.
\end{theorem}
This means that, for any $\Z \sg{n}$-module $M$, we may compute low-dimensional homology and cohomology by
\[H_i(\sg{n};M)=H_i(P_*\otimes_{\Z \sg{n}} M)\quad\text{ and }\quad H^i(\sg{n};M)=H^i(\hom_{\Z \sg{n}}(P_*,M))\]
for $i\in\{0,1,2\}$.

Because we will need to do explicit computations later, we define cocycles representing these cohomology classes.
Since $P_1$ and $P_2$ are free $\Z \sg{n}$-modules, to define our cocycles it suffices to define them on the bases given above.
For an abelian group $A$ and a positive integer $k$, $A[k]$ denotes the $k$-torsion of $A$.
\begin{definition}
    Let $A$ be an abelian group, considered as a $\Z \sg{n}$-module with the trivial action. 
    For $r\in A[2]$,
    define $\kappa^0_r\co P_1\to A$ by $\kappa^0_r(e_i)=r$ for all $i$ from $1$ to $n-1$.

    For $r\in A$, define $\alpha^0_r\co P_2\to A$ by $\alpha^0_r(c_i)=r$, $\alpha^0_r(d_{ij})=0$, and $\alpha^0_r(b_i)=0$, for all choices of $i$ and $j$ that make sense.

    Now assume $n\geq 4$.
    For $r\in A[2]$, define $\beta^0_r\co P_2\to A$ by
    $\beta^0_r(c_i)=0$, $\beta^0_r(d_{ij})=r$, and $\beta^0_r(b_i)=0$, for all choices of $i$ and $j$ that make sense.
\end{definition}
The reason for the superscript $0$ is to fit with notation that we use for Specht modules.

\begin{proposition}\label{pr:untwistedcohomreps}
    For all choices of $r$ as above, the functions $\kappa^0_r$, $\alpha^0_r$, and $\beta^0_r$ are cocycles.
    
    The map $A[2]\to H^1(\sg{n};A)$ by $r\mapsto [\kappa^0_r]$ is an isomorphism.

    If $n\geq 4$, then the map $A[2]\oplus A/2A\to H^2(\sg{n};A)$ by $(r,[s])\mapsto [\beta^0_r+\alpha^0_s]$ is an isomorphism.
    Similarly, if $n$ is $2$ or $3$, then $A/2A\to H^2(\sg{n};A)$ by $[s]\mapsto [\alpha^0_s]$ is an isomorphism.
\end{proposition}

\begin{proof}
    This is Proposition~4.7 from Day--Nakamura~\cite{DN3skeleton}.
\end{proof}

Now we fix a commutative ring $R$.
We adopt the convention that, if the group is suppressed in the cohomology notation, but only the module is  given, then it is cohomology of $\sg{n}$ in that module.
We also suppress the ring $R$ from the notation in the names of modules, unless we are dealing with multiple rings.

Again, we need to work explicitly, so we define cocycles that will represent our cohomology classes.
As above, we use $R[2]$ as notation for the $2$-torsion in $R$.
\begin{definition}
    For $r\in R[2]$, define $\kappa^1_r\co P_1\to M^1$ by $\kappa^1_r(e_i)=r\sum_{k=1}^nt_k$.

    For $r\in R$, define $\alpha^1_r\co P_2\to M^1$ by $\alpha^1_r(c_i)=r\sum_{k=1}^nt_k$, $\alpha^1_r(d_{ij})=0$, and $\alpha^1_r(b_i)=0$.
    
    Now assume $n\geq 4$. 
    For $r\in R[2]$, define $\beta^1_r\co P_2\to M^1$ by $\beta^1_r(c_i)=0$, $\beta^1_r(d_{ij})=r\sum_{k=1}^nt_k$, and $\beta^1_r(b_i)=0$.
\end{definition}

\begin{proposition}\label{pr:M1cohomology}
    For all choices of $r$ as above, the functions $\kappa^1_r$, $\alpha^1_r$, and $\beta^1_r$ are cocycles.

    We have $H^0(M^1)=R$, spanned by $\sum_{k=1}^nt_k$.

    If $n\geq 3$, we have $H^1(M^1)\cong R[2]$, and $R[2]\to H^1(M^1)$ by $r\mapsto [\kappa^1_r]$ is an isomorphism.  For $n=2$, $H^1(M^1)=0$.

    If $n\geq 5$, we have $H^2(M^1)\cong R[2]\oplus R/2R$, and $R[2]\oplus R/2R\to H^2(M^1_R)$ by $(r,[s])\mapsto [\alpha^1_s+\beta^1_r]$ is an isomorphism.
    For $n\in\{3,4\}$, $H^2(M^1)\cong R/2R$, and $R/2R\to H^2(M^1)$ by $[s]\mapsto[\alpha^1_s]$ is an isomorphism.
    If $n=2$, then $H^2(M^1)=0$.
\end{proposition}

\begin{proof}
    These functions are cocycles by Lemma~5.3 of Day--Nakamura~\cite{DN3skeleton}, where $k=1$ in that lemma.
    The statements about the first and second cohomology are Proposition~5.4 of~\cite{DN3skeleton}, with $k=1$.  
    The computation of $H^0(M^1)$ is a straightforward exercise.
\end{proof}

Now we do the same thing for $M^2$.
We recall that, as above, $u=\sum_{i<j}v_{ij}\in M^2$.
\begin{definition}
    For $r\in R[2]$, define $\kappa^2_r\co P_1\to M^2$ by $\kappa^2_r(e_i)=ru$.
    Define $\hat \kappa^2_r\co P_1\to M^2$ by $\hat\kappa^2_r(e_i)=rv_{i,i+1}$.
    
    For $r\in R$, define $\alpha^2_r\co P_2\to M^2$ by $\alpha^2_r(c_i)=ru$, $\alpha^2_r(b_i)=0$, and $\alpha^2_r(d_{ij})=0$.
    Define $\hat \alpha^2_r\co P_2\to M^2$ by $\hat\alpha^2_r(c_i)=rv_{i,i+1}$, $\hat\alpha^2_r(b_i)=0$, and $\hat\alpha^2_r(d_{ij})=0$.
    For $r\in R[2]$, define $\beta^2_r\co P_2\to M^2$ by $\beta^2_r(c_i)=0$, $\beta^2_r(b_i)=0$, and $\beta^2_r(d_{ij})=ru$.
    Define $\hat \beta^2_r\co P_2\to M^2$ by $\hat \beta^2_r(c_i)=0$, $\hat \beta^2_r(b_i)=0$, and
    \[\hat \beta^2_r(d_{ij})=r(v_{i,i+1}+v_{j,j+1}+v_{i,j}+v_{i,j+1}+v_{i+1,j}+v_{i+1,j+1}).\]
\end{definition}

\begin{proposition}\label{pr:M2cohomology}
For all choices of $r$ as above, the functions $\kappa^2_r$, $\hat\kappa^2_r$, $\alpha^2_r$, $\hat\alpha^2_r$, $\beta^2_r$, and $\hat\beta^2_r$ are cocycles.

    We have $H^0(M^2)=R$, spanned by $u$.
    
    If $n\geq4$, we have $H^1(M^2)\cong R[2]^2$, and $R[2]^2\to H^1(M^2)$ by $(r,s)\mapsto [\kappa^2_r+\hat\kappa^2_s]$ is an isomorphism.
    If $n\in \{2,3\}$, then $H^1(M^2)\cong R[2]$, and $r\mapsto [\kappa^2_r]$ is an isomorphism.

    If $n\geq6$, we have $H^2(M^2)\cong R[2]^2\oplus (R/2R)^2$, and $R[2]^2\oplus (R/2R)^2\to H^2(M^2)$ by $(q, r,[s],[t])\mapsto [\alpha^2_s+\hat\alpha^2_t+\beta^2_q+\hat\beta^2_r]$ is an isomorphism.
    For $n\in\{4,5\}$, we have $H^2(M^2)\cong R[2]\oplus (R/2R)^2$, and there is a similar isomorphism without the $\beta^2_q$.
    For $n\in \{2,3\},$ we have $H^2(M^2)\cong R/2R$, and there is an isomorphism by $r\mapsto [\alpha^2_r]$.
\end{proposition}

\begin{proof}
    Again, these statements follow from Lemma~5.3 and Proposition~5.4 of Day--Nakamura~\cite{DN3skeleton}.
    This time, we use $k=2$ in those statements.
    We also leave the computation of $H^0(M^2)$ as an exercise.
\end{proof}

\subsection{Extensions and the splitting obstruction}
We recall the connection between second cohomology and group extensions with abelian kernel.
Brown~\cite{Brown}, Chapter IV.3 is a reference.

For a group $G$ and a $\Z G$-module $M$, an \emph{extension} of $G$ by $M$ is a group $E$ with an exact sequence
\[0\to M\to E\to G\to 1,\]
such that the action of $E$ on $M$ by conjugation descends to the $G$-module structure on $M$.
Two extensions $E_1, E_2$ are \emph{equivalent} if they are isomorphic by an isomorphism fitting into a commuting diagram of extensions with the maps from $M$ and $G$ to themselves being the identity:
\[
\xymatrix{
     0 \ar[r] & M \ar[r] \ar@{=}[d] & E_1 \ar[r] \ar[d]^{\cong} & G \ar[r] \ar@{=}[d] & 1 \\   
     0 \ar[r] & M \ar[r] & E_2 \ar[r] & G \ar[r] & 1. \\   
}
\]
\begin{theorem}\label{th:EM}
Let $G$ be a group and let $M$ be a $\Z G$-module.
There is a canonical bijection from $H^2(G;M)$ to the set of equivalence classes of extensions 
\[0\to M\to E\to G\to 1.\]
There is a unique equivalence class of split extensions and this corresponds to $0\in H^2(G;M)$.
\end{theorem}
Given an extension with cokernel $G$ and kernel $M$, we call the corresponding element of $H^2(G;M)$ its \emph{structure class}.

\begin{remark}
    We like to attribute Theorem~\ref{th:EM} to Eilenberg and MacLane, but it appears that the structure of extensions with abelian kernel was already well understood at the time that Eilenberg and MacLane defined cohomology of groups.  
    The generalization of Theorem~\ref{th:EM} to extensions with nonabelian kernel is due to Eilenberg--MacLane in~\cite{EM}.
\end{remark}

There is one aspect of this theorem that we need to make concrete: given an extension, we can find a representative of its structure class by using a section.
If $\pi\co E\to G$ is the projection from an extension, a \emph{section} to $\pi$ is a set-map $s\co G\to E$ with $\pi\circ s=\mathrm{id}_G$.
We also need the \emph{bar chains} for group homology; let $C_k(G)$ be the free $\Z G$-module with basis $G^k=G\times \dotsm\times G$.
A basis element is denoted $[g_1|\dotsm|g_k]$.
We will not be using the boundary map, so we do not define it here.
The following proposition is implicit in the exposition in Brown~\cite{Brown}:
\begin{proposition}\label{pr:EMexplicit}
    Let $0\to M\to E\to G\to 1$ be an extension and let $s\co G\to E$ be a section to the projection $E\to G$ (not necessarily a homomorphism).
    Let $c\co C_2(G)\to M$ be defined by the formula
    \[s(g)s(h)=i(c([g|h]))s(gh),\]
    for $g,h\in G$, where $i$ is the inclusion map $i\co M\to E$.
    Then $c$ is a cocycle, $[c]\in H^2(G;A)$ does not depend on the choice of $s$, and $[c]$ is the structure class of this extension.
\end{proposition}
We need the following corollary:
\begin{corollary}\label{co:pushfowardstructure}
    Let $M_1$ and $M_2$ be $\Z G$-modules and let $f\co M_1\to M_2$  be $\Z G$-linear.
    Suppose that $E_1$ and $E_2$ are extensions of $G$ by $M_1$ and $M_2$ respectively, and let $\psi\co E_1\to E_2$ be a homomorphism, such that the following diagram commutes:
    \[
\xymatrix{
     0 \ar[r] & M_1 \ar[r] \ar[d]^f & E_1 \ar[r] \ar[d]^{\psi} & G \ar[r] \ar@{=}[d] & 1 \\   
     0 \ar[r] & M_2 \ar[r] & E_2 \ar[r] & G \ar[r] & 1. \\   
}
\]
    Let $\kappa\in H^2(G;M_1)$ denote the structure class of $M_1\to E_1\to G$.
    Then $f_*\kappa\in H^2(G;M_2)$  is the structure class of $M_2\to E_2\to G$.
\end{corollary}

\begin{proof}
    Let $s\co G\to E_1$ be a section to the projection $E_1\to G$.
    Then since the diagram commutes, $\psi\circ s\co G\to E_2$ is a section to the projection $E_2\to G$.
    Let $i_1\co M_1\to i_1(M_1)\subseteq E_1$ and $i_2\co M_2\to i_2(M_2)\subseteq E_2$ be the restrictions of the inclusions.
Let $c\co C_2(G)\to M_1$ be given by $c([g|h])=i_1^{-1}(s(g)s(h)s(gh)^{-1})$; then by Proposition~\ref{pr:EMexplicit}, we have $[c]=\kappa$.
Since the diagram commutes, we have $i_2\circ f=\psi\circ i_1$, and therefore $f\circ i_1^{-1}=i_2^{-1}\circ \psi$.
The class $f_*\kappa$ is represented by $f\circ c$, but for $[g|h]\in C_2(G)$, we have
\[
\begin{split}
f(c([g|h]))&=f(i_1^{-1}(s(g)s(h)s(gh)^{-1}))=i_2^{-1}(\psi(s(g)s(h)s(gh)^{-1}))\\
&=i_2^{-1}(\psi(s(g))\psi(s(h))\psi(s(gh))^{-1}).
\end{split}
\]
Again by Proposition~\ref{pr:EMexplicit}, since $\psi\circ s$ is a section, this says that $f_*\kappa$ corresponds to the second extension.
\end{proof}

The extension from Proposition~\ref{pr:recognizepairmodule}
\begin{equation}\label{eq:initialextension}
0\to M^2_\Z\to \frac{B_n}{PB_n'}\to \sg{n}\to 1    
\end{equation}
is initial in the category of extensions of the symmetric group that are quotients of the braid group.
So to understand these extensions better, we should start with this one.
This was one of the topics of the previous paper of the second author~\cite{Nakamura}.
However, the results there were given in terms of a different presentation of the symmetric group, so we need to take a moment to translate.
The other presentation gives us a partial free resolution $R_*$ of $\Z$ as a $\Z\sg{n}$-module, defined in dimensions $0$, $1$, and $2$, so that we can give elements of $H^2(\sg{n};M)$ for any module $M$ using $\hom(R_2,M)$.
Let $R_*$ be the complex of free $\Z \sg{n}$-modules defined by:
\begin{itemize}
    \item $R_0$ is free on a single element $*$;
    \item $R_1$ has a free basis consisting of $\binom{n}{2}$ elements $\tilde x_{ij}$ for $1\leq i<j\leq n$, with $\partial \tilde x_{ij}=((ij)-1)*$ (here $(ij)\in \sg{n}$ is the transposition),
    \item and $R_2$ has a free basis consisting of three classes of elements:
    \begin{itemize}
        \item $\binom{n}{2}$ elements $\tilde c_{ij}$, for $1\leq i<j\leq n$, with $\partial \tilde c_{ij}=((ij)+1)x_{ij}$, 
        \item $\binom{n}{2}\binom{n-2}{2}$ elements $\tilde d_{ijkl}$ for $1\leq i<j\leq n$, $1\leq k<l\leq n$, $\{i,j\}\cap\{k,l\}=\varnothing$, with
        \[\partial \tilde d_{ijkl}=(1-(kl))\tilde x_{ij}-(1-(ij))\tilde x_{kl},\]
        \item and $n(n-1)(n-2)$ elements $\tilde e_{ikj}$ for $1\leq i,j,k\leq n$ with $i,j,k$ distinct, and 
        \[\partial \tilde e_{ikj}=(ij)\tilde x_{jk}+(1-(ik))\tilde x_{ij}-\tilde x_{ik}.\]
    \end{itemize}
\end{itemize}

    This is the same as the definition given in Section~3 of Nakamura~\cite{Nakamura}.
We can now state one of the main results of~\cite{Nakamura}, Theorem~4.10 of that paper.
(We translate to this paper's notation for $M^2_\Z$, but the formulas are the same.)
\begin{theorem}[Nakamura~\cite{Nakamura}]\label{th:Nakamura}
Let $\phi\co R_2\to M^2_\Z$ be the $\Z \sg{n}$-module homomorphism defined by
\[\phi(\tilde c_{ij})=v_{ij},\]
\[\phi(\tilde d_{ijkl})=
\left\{
\begin{array}{cl}
v_{ik}-v_{il}-v_{kj}+v_{jl} & \text{if $i<k<j<l$,} \\
-v_{ik}+v_{kj}+v_{il}-v_{lj} & \text{if $k<i<l<j$,} \\
0 & \text{otherwise,}
\end{array}
\right.
\]
\[
\phi(\tilde e_{ikj})=
\left\{
\begin{array}{cl}
v_{ij}-v_{kj} & \text{if $i<k<j$, or $j<i<k$, or $k<j<i$,} \\
0 & \text{otherwise.}
\end{array}
\right.
\]
Then $\phi$ is a cocycle, and $[\phi]\in H^2(\sg{n};M^2_\Z)$ is the structure class of the extension~\eqref{eq:initialextension}.
\end{theorem}

We state the following corollary to highlight the importance of $[\phi]$ to the current paper.
In short, the push-forward of $[\phi]$ is the splitting obstruction for an extension of the type that we are studying.
\begin{corollary}\label{co:splittingobstruction}
    Let $N\lhd B_n$ with $PB_n'\leq N\leq PB_n$, and let $M=PB_n/N$, as a $\Z\sg{n}$-module.
    Let $\pi\co M^2_\Z\to M$ be the quotient projection (identifying $M^2_\Z$ with $PB_n/PB_n'$ via Proposition~\ref{pr:recognizepairmodule}).
    Then $\pi_*[\phi]\in H^2(\sg{n};M)$ is the structure class of the extension
    \[0\to M\to B_n/N\to \sg{n}\to 1,\]
    and therefore this extension splits if and only $\pi_*[\phi]=0$.
\end{corollary}

\begin{proof}
The point is that we have a commuting diagram
      \[
\xymatrix{
     0 \ar[r] & M^2_\Z \ar[r] \ar[d]^\pi & B_n/PB_n' \ar[r] \ar[d] & \sg{n} \ar[r] \ar@{=}[d] & 1 \\   
     0 \ar[r] & M \ar[r] & B_n/N \ar[r] & \sg{n} \ar[r] & 1, \\   
}
\] 
so that Corollary~\ref{co:pushfowardstructure} applies.
Since $[\phi]$ is the structure class of the first extension by Theorem~\ref{th:Nakamura}, this proves the corollary.
\end{proof}

We need to transfer the statement of Theorem~\ref{th:Nakamura} over to the chain complex that we are using in the current paper.
After proving this result, we use $P_*$ only (and not $R_*$) for the rest of the paper.
\begin{proposition}\label{pr:Nakamuracocycle}
There is an augmentation-preserving partial chain homotopy equivalence $\psi\co P_*\to R_*$, such that $\phi\circ\psi\co P_2\to M^2_\Z$ is the map $c_i\mapsto v_{i,i+1}$, $d_{ij}\mapsto 0$, $b_i\mapsto 0$.
In particular $\phi\circ\psi$ is a cocycle using $P_*$
with $[\phi\circ\psi]=[\hat\alpha^2_1]\in H^2(\sg{n};M^2_\Z)$.
So $[\hat\alpha^2_1]$ is the structure class of the extension~\eqref{eq:initialextension}.
\end{proposition}

\begin{proof}
%[Proof of Proposition~\ref{pr:Nakamuracocycle}]
    The following formulas work to define $\psi$:
    \begin{itemize}
        \item $\psi(*)=*$,
        \item $\psi(e_i)=\tilde x_{i,i+1}$,
        \item $\psi(c_i)=\tilde c_{i,i+1}$,
        \item $\psi(d_{ij})=\tilde d_{i,i+1,j,j+1}$,
        \item $\psi(b_i)=\tilde e_{i+2,i,i+1}-\tilde e_{i,i+2,i+1}-s_is_{i+1}\tilde c_{i,i+1}+s_{i+1}s_i\tilde c_{i+1,i+2}$.
    \end{itemize}
    (Pictorially, $b_i$ is a hexagon, and $\psi$ tiles it with two $\tilde e$-quadrilaterals and two $\tilde c$-bigons.  The bigons flip the directions on two of the edges so that they match between the hexagon and the quadrilaterals.)
    
    Let $\epsilon_P\co P_0\to \Z$ and $\epsilon_R\co R_0\to \Z$ be the augmentation maps $*\mapsto 1$.
    Then $\epsilon_R\circ \psi=\psi\circ \epsilon_P$.
    Also, for each chain $c$ in $P_1$ or $P_2$, we have $\psi(\partial c)=\partial \psi(c)$ (by direct computation).
    So $\psi$ is an augmentation-preserving chain map, and it follows from Theorem I.7.5 of Brown~\cite{Brown} that $\psi$ is a chain homotopy equivalence, up to dimension $2$.
    (In more detail: if $P_*$ and $R_*$ are independently completed to free resolutions of $\Z$, then there is an extension of $\psi$ that is an augmentation-preserving chain map between them, and this extension is a homotopy equivalence.)

    Now we evaluate $\phi\circ\psi$.
    It is easy to see that $\phi(\psi(c_i))=v_{i,i+1}$.
    Since $i<i+1<j<j+1$ (by our conventions on $d_{ij}$), we have that $\phi(\psi(d_{ij}))=0$.
    Since $i<i+1<i+2$, but $(i,i+2,i+1)$ is in the reverse cyclic order, we have that $\phi(\tilde e_{i,i+2,i+1})=0$ and $\phi(e_{i+2,i,i+1})=v_{i+1,i+2}-v_{i,i+1}$.
    But $\phi(s_is_{i+1}\tilde c_{i,i+1})=v_{i+1,i+2}$ and $\phi(s_{i+1}s_i\tilde c_{i+1,i+2})=v_{i,i+1}$.
    So $\phi(\psi(b_i))=0$.
\end{proof}

One of the goals of the current paper is to provide context for the second author's results in~\cite{Nakamura}, in which he describes the structure class for extension~\eqref{eq:initialextension} and the extension for the quotient by the level-4 subgroup:
\begin{equation}\label{eq:level4ext}
0\to M^2_{\Z/2\Z}\to B_n/B_n[4] \to \sg{n}\to 1.
\end{equation}
Having described the cohomology in $M^2_R$ for general $R$, we can provide that context now.
The following corollary is an immediately consequence of Proposition~\ref{pr:M2cohomology}, and uses Theorem~\ref{th:EM} and Theorem~\ref{th:Nakamura}.
It needs no further proof:

\begin{corollary}
    We have that $H^2(M^2_\Z)=(\Z/2\Z)^2$, generated by $[\alpha^2_1]$ and $[\hat\alpha^2_1]$.  
    So there are exactly four equivalence classes of extensions with cokernel $\sg{n}$ and kernel $M^2_\Z$, and the one with structure class $[\hat \alpha^2_1]$ is the class of extension~\eqref{eq:initialextension}.
\end{corollary}

\begin{corollary}
    Let $R=\Z/2\Z$.
    We have that $H^2(M^2_R)=R^4$, generated by $[\alpha^2_1]$, $[\hat\alpha^2_1]$, $[\beta^2_1]$, and $[\hat\beta^2_1]$.  
    There are exactly sixteen equivalence classes of extensions with cokernel $\sg{n}$ and kernel $M^2_R$, and the one with structure class $[\hat \alpha^2_1]$ is the class of extension~\eqref{eq:level4ext}.
\end{corollary}

\begin{proof}
    All we need to explain is why this is the right description of the structure class.
    The point is that the projection $B_n/PB_n/\to B_n/B_n[4]$ restricts to the map $M^2_\Z\to M^2_R$ that reduces the coefficients modulo $2$.  
    This sends the structure class $[\hat\alpha^2_1]\in H^2(M^2_\Z)$ to the class $[\hat\alpha^2_1]\in H^2(M^2_R)$.
    Corollary~\ref{co:pushfowardstructure} then implies that this is the structure class for extension~\eqref{eq:level4ext}.
\end{proof}

\section{Cohomology computations}
\subsection{Cohomology in $S^1$ and $S^2$}
The modules $M^1$ and $S^1$ fit into the following exact sequence
\begin{equation}\label{eq:S1seq}
    0\to S^1\to M^1\to R\to 0;
\end{equation}
the second map is $t_i\mapsto 1$.
Since $S^1$ is defined to be spanned by differences $t_i-t_j$, this sequence is exact.
We use this to find the cohomology of $\sg{n}$ with coefficients in $S^1$.
\begin{proposition}\label{pr:S1cohomology}
    Let $n\geq 2$.
    We have
    $H^0(S^1)\cong R[n],$
    the $n$-torsion.  
    More specifically, it is equal to the submodule $\{r \sum_{i=1}^nt_i\,|\, r\in R[n]\}.$
    
    We have
    \[
    H^1(S^1)=\left\{
    \begin{array}{cl}
    R/nR & \text{if $n$ is odd or $n=2$}\\
    (R/nR)\oplus R[2] & \text{ if $n$ is even and $n\geq 4$}
    \end{array}
    \right.
    \]
    The $R[2]$-factor consists of elements of the form $[\kappa^1_r]$, and the $R/nR$-factor is the image of $H^0(R)$ under the connecting homomorphism from sequence~\eqref{eq:S1seq}.

    We have
    \[H^2(S^1)=\left\{\begin{array}{cl}
    0 & \text{if $n$ is odd} \\
    R[2]^2\oplus R/2R & \text{if $n$ is even and $n\geq 6$} \\
    R[2]\oplus R/2R & \text{if $n=4$} \\
    R[2] & \text{if $n=2$.} \\
    \end{array}\right.
    \]
    The $R/2R$-factor consists of elements that map to elements of the form $[\alpha^1_r]\in H^2(M^1)$, the $R[2]$-factor starting at $n=6$ consists of elements that map to elements of the form $[\beta^1_r]$ in $H^2(M^1)$, and the $R[2]$-factor starting at $n=2$ is the image of $H^1(R)$ under the connecting homomorphism from the sequence~\eqref{eq:S1seq}.
\end{proposition}

\begin{proof}
    The exact sequence~\eqref{eq:S1seq} gives us the following long exact sequence for cohomology:
    \[
        \begin{split}
            0&\to H^0(S^1)\to H^0(M^1) \to H^0(R) \to H^1(S^1) \to H^1(M^1)\to H^1(R) \\
            &\to H^2(S^1)\to H^2(M^1)\to H^2(R)\to \dotsm
        \end{split}
    \]
    We know that $H^0(M^1)$ is spanned by $\sum_{i=1}^nt_i$, and this element maps to $n$ in $H^0(R)=R$.
    The statement about $H^0(S^1)$ follows immediately.

    Next we compute $H^1(S^1)$.
    We start by considering the connecting homomorphism $H^0(R)\to H^1(S^1)$.
    By exactness, the kernel of this map is the image of the map $H^0(M^1)\to H^0(R)$, which is $nR$.
    So the image of the connecting homomorphism is a submodule of $H^1(S^1)$ isomorphic to $R/nR$.
    (The element $1\in R$ lifts to $t_1\in M^1$, and the coboundary of $t_1$, considered as a $0$-cocycle, is the map given by $e_i\mapsto (s_i-1)t_1\in S^1$.
    The $n$-multiple of this cocycle represents the same class as $e_i\mapsto (s_i-1)\sum_{j=1}^nt_j$, which is the zero cocycle.)

    We also need to consider homomorphism $H^1(M^1)\to H^1(R)$.
    By Proposition~\ref{pr:M1cohomology}, $H^1(M^1)$ consists of classes of the 
    form $[\kappa^1_r]$ for $r\in R[2]$, if $n\geq 3$.  (If $n=2$, $H^1(M^1)=0$.)
    The class $[\kappa^1_r]$ maps to a class represented by $e_i\mapsto nr$.
    Since $r$ has order $2$, we know that $nr=0$ if $n$ is even, and $nr=r$ if $n$ is odd.
    So if $n$ is even, every $[\kappa^1_r]$ is in the kernel, and therefore these are classes in $H^1(S^1)$.
    (If $n$ is even, then since $r$ has order $2$, $r\sum_{i=1}^nt_i=r((t_1-t_2)+(t_3-t_4)+\dotsm+(t_{n-1}-t_n))$, which is in $S^1$.)
    Since $2\kappa^1_r$ is the zero function, considered as a function to $S^1$ or to $M^1$, we have that $2[\kappa^1_r]=0$, and therefore $H^1(S^1)$ decomposes as a direct sum.
    Of course, if $n$ is odd, the kernel of $H^1(M^1)\to H^1(R)$ contributes nothing, and this proves the statement about $H^1(S^1)$.
    
    Now we compute $H^2(S^1)$.
    We just considered the map $H^1(M^1)\to H^1(R)$, and it is trivial if $n$ is even, and by Proposition~\ref{pr:untwistedcohomreps}, it is surjective if $n$ is odd.
    The image of this map is the kernel of the connecting homomorphism $H^1(R)\to H^2(S^1)$, so if $n$ is odd, $H^1(R)$ contributes nothing to $H^2(S^1)$, and if $n$ is even, then $H^1(R)$ maps injectively to $H^2(S^1)$.
    This contributes a submodule of the form $R[2]$ starting at $n=2$.

    Finally, we consider the map $H^2(M^1)\to H^2(R)$.
    For a generator $[\alpha^1_r]$, we have that $\alpha^1_r$ maps to the map $c_i\mapsto nr$, with other cells mapping to $0$.
    Again, this is $0$ if $n$ is even.
    If $n$ is odd, then this is $[\alpha^0_r]$, since $nr=r$.
    For a generator $[\beta^1_r]$, we have that $\beta^1_r$ maps to the map $d_{ij}\mapsto nr$, with other cells mapping to $0$.
    Again, if $n$ is even this is $0$, and if $n$ is odd, this is $\beta^0_r$.
    So $H^2(M^1)\to H^2(R)$ is injective if $n$ is odd, and it is the zero map if $n$ is even.
    If $n$ is even, then a cocycle $\alpha^1_r$ lifts to to $e_i\mapsto r((t_1-t_2)+(t_3-t_4)+\dotsm+(t_{n-1}-t_n))$.
    Then $2\alpha^1_r$ is the coboundary of $s_i\mapsto r((t_1-t_2)+(t_3-t_4)+\dotsm+(t_{n-1}-t_n))$. 
    In particular, a class $[\alpha^1_r]$ or $[\beta^1_r]$ has order $2$ when considered in $H^2(S^1)$, so $H^2(S^1)$ decomposes as a direct sum of the image and the kernel of the adjacent maps in the exact sequence.
    The statement about $H^2(S^1)$ follows immediately.
\end{proof}

Now our goal is to prove a similar statement for $S^2$.
To do this, we need to consider cohomology in the submodule $K_{12}$ of $M^2$.
We recall two exact sequences, one from the statement of Proposition~\ref{pr:SESS2K12S1}, and one from the discussion before that proposition.
They are
\begin{equation}\label{eq:K12seq}
    0\to K_{12}\to M^2\to R\to 0,
\end{equation}
and 
\begin{equation}\label{eq:S2seq}
    0\to S^2\to K_{12}\to S^1 \to 0.
\end{equation}
The map $M^2\to R$ is $v_{ij}\mapsto 1$, and sequence~\eqref{eq:K12seq} defines the submodule $K_{12}$.
The map $K_{12}\to S^1$ is the restriction of $\mu\co S^2\to S^1$, given by $\mu(v_{ij})=t_i+t_j$.
Also recall that for each $i$, $w_i\in M^2$ is given by $w_i=\sum_{j\neq i}v_{ij}$, and $u\in M^2$ is given by $u=\sum_{i,j}v_{ij}$.
In the following proposition, we don't consider the case $n=2$ since if $n=2$, then $K_{12}=0$.
\begin{proposition}\label{pr:Kcohomolgy}
    Let $n\geq 3$.    
    We have
    $H^0(K_{12})\cong R[\binom{n}{2}],$
    the $\binom{n}{2}$-torsion.  
     Specifically, it is equal to the submodule $\{ru \,|\, r\in R[\binom{n}{2}]\}.$

    We have
    \[
    H^1(K_{12})=\left\{
    \begin{array}{cl}
    R/3R & \text{if $n=3$}\\
    R[2]\oplus R/\binom{n}{2}R & \text{ if $n\geq 4$}
    \end{array}
    \right.
    \]
    The $R[2]$-factor consists of elements of the form $[\kappa^2_r]$ if $\binom{n}{2}$ is even, and elements of the form $[\kappa^2_r-\hat\kappa^2_r]$ if $\binom{n}{2}$ is odd.
    The $R/\binom{n}{2}R$-factor is the image of $H^0(R)$ under the connecting homomorphism from sequence~\eqref{eq:K12seq}.

    We have
    \[H^2(K_{12})=\left\{\begin{array}{cl}
    0 & \text{if $n=3$} \\
    %R[2]\oplus R/2R & \text{if $n=4$ or $n=5$} \\
    R[2]\oplus R/2R & \text{if $n\in\{4,5\}$, or if $n\geq 6$ and $\binom{n}{2}$ is odd} \\
    R[2]^2\oplus R/2R & \text{if $n\geq 6$ and $\binom{n}{2}$ is even.} \\
    \end{array}\right.
    \]
    The $R/2R$-factor consists of elements of the form $[\alpha^2_r]$ if $\binom{n}{2}$ is even, and of the form $[\alpha^2_r-\hat\alpha^2_r]$ if $\binom{n}{2}$ is odd.
    One  of the $R[2]$-factors consists of elements of the form $[\hat\beta^2_r]$.
    If $n\geq 6$ and $\binom{n}{2}$ is even, then the other $R[2]$-factor consists of elements of the form $[\beta^2_r]$.
\end{proposition}

\begin{proof}
    We need the beginning of the cohomology long exact sequence from the sequence~\eqref{eq:K12seq}:
    \[
        \begin{split}
            0&\to H^0(K_{12})\to H^0(M^2) \to H^0(R) \to H^1(K_{12}) \to H^1(M^2)\to H^1(R) \\
            &\to H^2(K_{12})\to H^2(M^2)\to H^2(R)\to \dotsm
        \end{split}
   \]
    The map $H^0(M^2)\to H^0(R)$ sends the generator $u$ to $\binom{n}{2}$, and this proves the statement about $H^0(K_{12})$.
    This image is a copy of $\binom{n}{2}R$, and is equal to the kernel of $H^0(R)\to H^1(K_{12})$.
    So this map contributes a submodule of $H^1(K_{12})$ isomorphic to $R/\binom{n}{2}R$.

    Next consider the map $H^1(M^2)\to H^1(R)$, and temporarily assume $n\geq 4$.
    By Proposition~\ref{pr:M2cohomology}, we know that $H^1(M^2)$ is generated by classes of the form $[\kappa^2_r]$ and $[\hat\kappa^2_r]$.
    We see that $[\kappa^2_r]$ maps to $[\binom{n}{2}\kappa^0_{r}]$, and $[\hat\kappa^2_r]$ maps to $[\kappa^0_r]$.
    These classes are defined for $r\in R[2]$, so $\binom{n}{2}r$ is $r$ if $\binom{n}{2}$ is odd, and is $0$ if $\binom{n}{2}$ is even.
    So if $\binom{n}{2}$ is even, the kernel consists of classes of the form $[\kappa^2_r]$, and if $\binom{n}{2}$ is odd, the kernel consists of classes of the form $[\kappa^2_r-\hat\kappa^2_r]$.
    These representative cocycles have order $2$, so $H^1(K_{12})$ decomposes as a direct sum.
    If $n=3$, then $H^1(M^2)$ consists only of classes of the form $[\kappa^2_r]$, and the kernel of $H^1(M^2)\to H^1(R)$ is trivial.
    
    By Proposition~\ref{pr:untwistedcohomreps}, we know $H^1(R)$ consists of classes of the form $[\kappa^0_r]$.
    Since $[\hat\kappa^2_r]$ always hits this class, the map $H^1(M^2)\to H^1(R)$ is surjective, and $H^1(R)$ contributes nothing to $H^2(K_{12})$.

    Finally we compute the kernel of $H^2(M^2)\to H^2(R)$.
    Temporarily assume that $n\geq 6$.
    We use the representatives from Proposition~\ref{pr:M2cohomology}.
    We compute that $[\alpha^2_r]\mapsto [\binom{n}{2}\alpha^0_{r}]$, 
    $[\hat\alpha^2_r]\mapsto[\alpha^0_r]$, $[\beta^2_r]\mapsto[\binom{n}{2}\beta^0_{r}]$, and $[\hat\beta^2_r]\mapsto [\beta^0_{6r}]=0$ (since $r\in R[2]$ if $\beta^2_r$ is defined).
    First suppose that $\binom{n}{2}$ is even.
    Then the kernel is generated by elements of the form $[\alpha^2_r]$, $[\beta^2_r]$, and $[\hat\beta^2_r]$.
    (We see that $[\alpha^2_r]$ maps to a multiple of $2[\alpha^0_r]=0$).
    Next suppose that $\binom{n}{2}$ is odd.
    The kernel is generated by elements of the form $[\alpha^2_r-\hat\alpha^2_r]$ and $[\hat\beta^2_r]$.
    If $n$ is $4$ or $5$, then we leave out the classes of the form $[\beta^2_r]$, but our description of the kernel is otherwise the same as if $n\geq 6$ and $\binom{n}{2}$ is even.
    If $n=3$, then $H^2(M^2)$ only has classes of the form $[\alpha^2_r]$, and the kernel is trivial.
    This proves the description of $H^2(K_{12})$ in the statement.
\end{proof}

\begin{remark}\label{re:K12cocycles}
    Above, we give representatives for cohomology classes in the kernel of $H^1(M^2)\to H^1(R)$.
    These cocycles actually map to $K_{12}$, and there is a nice way to see this.
    If $\binom{n}{2}$ is even, write $u=x+y$ where each of $x$ and $y$ is a sum of $\frac{1}{2}\binom{n}{2}$ basis vectors.  Since $r\in R[2]$, we can think of $\kappa^2_r$ as sending $s_i$ to $r(x-y)$, and $r(x-y)\in K_{12}$.
    If $\binom{n}{2}$ is odd, then for each $i$, we can write $u-v_{i,i+1}$ as $x+y$ where each of $x$ and $y$ is a sum of $\frac{1}{2}(\binom{n}{2}-1)$ basis vectors.
    We then recognize $(\kappa^2_r-\hat\kappa^2_r)(s_i)=r(x-y)\in K_{12}$.

    We also give representatives for the classes in the kernel of $H^2(M^2)\to H^2(R)$.
    We can use tricks similar to what we just did to show that $\hat\beta^2_r$ maps to $K_{12}$, and $\beta^2_r$ if $\binom{n}{2}$ is even, but $\alpha^2_r$ and $\alpha^2_r-\hat\alpha^2_r$ don't actually map to $K_{12}$ for some choices of $r$.  
    The trick with these is to correct the cocycle by a coboundary to get the image to land in $K_{12}$.
    Specifically, consider the function $\hat\kappa^2_r$, but without the requirement that $r\in R[2]$.
    Then $\partial\hat\kappa^2_r=2\hat\alpha^2_r$.
    If $\binom{n}{2}$ is even, then $\alpha^2_r-\binom{n}{2}\hat\alpha^2_r$ is in the same class as $\alpha^2_r$, but maps to $K_{12}$.
    If $\binom{n}{2}$ is odd, then $\alpha^2_r-\binom{n}{2}\hat\alpha^2_r$ is in the same class as $\alpha^2_r-\hat\alpha^2_r$, but maps to $K_{12}$.
\end{remark}

In the following theorem, we assume that $n\geq 4$ since $S^2$ is trivial if $n<4$.
As in Section~\ref{se:ssquot}, let $b$ be $1$ if $n$ even and $1/2$ if $n$ is odd.
\begin{theorem}\label{th:S2cohomology}
    Let $n\geq 4$.  Then $H^0(S^2)\cong R[b(n-1)]$,
    %\[H^0(S^2)\cong R[\gcd(n-1,\binom{n}{2})]=
    %\left\{
    %\begin{array}{cl}
    %R[n-1] & \text{if $n$ is even} \\
    %R[\frac{n-1}{2}] & \text{if $n$ is odd,}
    %\end{array}
    %\right.
    %\]
    and as a submodule of $S^2$, it is $\{ru\,|\,r\in R[b(n-1)]\}$.

    If $n$ is odd, we have
    \[H^1(S^2)=R[2]\oplus \ker(\frac{R}{\binom{n}{2}R}\stackrel{2}{\longrightarrow}\frac{R}{nR}).\]
    The $R[2]$-factor consists of classes of the form $[\kappa^2_r]$ if $\binom{n}{2}$ is even, and $[\kappa^2_r-\hat\kappa^2_r]$ if $\binom{n}{2}$ is odd.
    The factor that is a kernel inside $R/\binom{n}{2}R$ consists of classes coming from the connecting homomorphism from sequence~\eqref{eq:S1seq}.
    
    If $n$ is even, then $H^1(S^2)$ fits in a short exact sequence
    \[0\to \frac{R[n]}{(n-1)(R[\binom{n}{2}])} \to H^1(S^2) \to  \ker(\frac{R}{\binom{n}{2}R}\stackrel{2}{\longrightarrow}\frac{R}{nR})\to 0\]
    The factor of the form $R[n]/(n-1)(R[\binom{n}{2}])$ comes from the connecting homomorphism from sequence~\eqref{eq:S2seq}, and the other factor again comes from the connecting homomorphism from sequence~\eqref{eq:S1seq}.

    \[
    H^2(S^2)= 
    \left\{
    \begin{array}{cl}
    R/2R & \text{if $n\equiv 0 \pmod{4}$} \\
    R[2]^2\oplus R/2R & \text{if $n\equiv 1 \pmod{4}$ and $n\geq 9$} \\
    R[2]\oplus R/2R & \text{if $n\equiv 2 \pmod{4}$, $n\equiv 3 \pmod{4}$, or $n=5$.}
    \end{array}
    \right.
    \]
    If $n$ is odd, then the inclusion $S^2\to K_{12}$ induces an isomorphism $H^2(S^2)\to H^2(K_{12})$, and we can find representatives by modifying our representatives for $H^2(K_{12})$.
    If $n$ is even, then the $R/2R$-factor comes from the connecting homomorphisms of both sequence~\eqref{eq:S1seq} and sequence~\eqref{eq:S2seq}, and the $R[2]$-factor consists of classes of the form $[\hat\beta^2_r]$.
\end{theorem}

\begin{proof}
    We write out the beginning of the long exact sequence for cohomology in the sequence~\eqref{eq:S2seq}:
        \[
        \begin{split}
            0&\to H^0(S^2)\to H^0(K_{12}) \to H^0(S^1) \to H^1(S^2) \to H^1(K_{12})\to H^1(S^1) \\
            &\to H^2(S^2)\to H^2(K_{12})\to H^2(S^1)\to \dotsm
        \end{split}
   \]
   From Propositions~\ref{pr:S1cohomology} and~\ref{pr:Kcohomolgy}, we know that $H^0(K_{12})=R[\binom{n}{2}]$, and $H^0(S^1)=R[n]$.
   The map $K_{12}\to S^1$ is the restriction of $\mu\co v_{ij}\mapsto t_i+t_j$.
   An element of $H^0(K_{12})$ has the form $ru$ with $r\in R[\binom{n}{2}]$.
   We compute that $\mu(u)=(n-1)\sum_{i=1}^n t_i$.  (To see this, note that each $i$ appears in exactly $(n-1)$ of the $v_{ij}$.)
   So $ru\in H^0(K_{12})$ maps to $(n-1)r\sum_{i=1}^nt_i\in H^0(S^1)$.
   This is the map $R[\binom{n}{2}]\to R[n]$ given by scaling by $n-1$.
   Its kernel is then $R[\binom{n}{2}]\cap R[n-1]$.
   We note that
   \[R[\binom{n}{2}]\cap R[n-1]=R[\gcd(\binom{n}{2},n-1)]=R[b(n-1)].\]
   This proves the statement about $H^0(S^2)$.

   Now we consider the cokernel of the map $H^0(K_{12})\to H^0(S^1)$.
   From the last paragraph, this can be described as $R[n]/(n-1)(R[\binom{n}{2}])$.
   Momentarily suppose $n$ is odd, and let $r\in R[n]$.
   Since $n$ is odd, $n$ divides $\binom{n}{2}$, and therefore $r\in R[\binom{n}{2}]$.
   Since $r$ has order $n$, $(n-1)r=-r$.
   This implies that scaling by $n-1$ maps surjectively from $R[\binom{n}{2}]$ to $R[n]$ if $n$ is odd.
   Since the cokernel is trivial in this case, and we omit it from our description of $H^1(S^2)$.
   If $n$ is even, then this quotient can be nontrivial (for example, if $R=\Z/n\Z$), and this cokernel is a submodule of $H^1(S^2)$.
   
   Next we consider the map $\mu_*\co H^1(K_{12})\to H^1(S^1)$.
   The plan is to use Propositions~\ref{pr:S1cohomology} and~\ref{pr:Kcohomolgy}, and describe the map in four cases, depending on the congruence class of $n$ modulo $4$.
   In all cases, $H^1(K_{12})$ is isomorphic to $R/\binom{n}{2}R\oplus R[2]$, but the generators depend on the parity of $\binom{n}{2}$.
   We need two connecting homomorphisms, which we temporarily name $\delta_1\co H^0(R)\to H^1(S^1)$ and $\delta_2\co H^0(R)\to H^1(K_{12})$.
   We compute $\mu_*$ on three kinds of cocycles ahead of this.
   First of all, $\delta_2(1)$ is represented by $s_i\mapsto (s_i-1)v_{12}$.
   This maps to $s_i\mapsto (s_i-1)(t_1+t_2)$, which we recognize as representing $2\delta_1(1)$.
   We also consider the cocycles $\kappa^2_r$ and $\hat\kappa^2_r$, for $r\in R[2]$.
   Since $\mu(u)=(n-1)\sum_{i=1}^nt_i$, we have $\mu_*\kappa^2_r=(n-1)\kappa^1_r$.
   Evaluating $\mu_*\hat\kappa^2_r$ gives the cocycle $s_i\mapsto r(t_i+t_{i+1})$.
   Since $r$ has order $2$, we can recognize this as the coboundary in $M^1$ of $rt_{\mathrm{odd}}$, where
   \[t_{\mathrm{odd}}=t_1+t_3+t_5+\dotsm+t_{2\lfloor\frac{n-1}{2}\rfloor+1}.\]
    Since there are $\lfloor\frac{n-1}{2}\rfloor+1$ many terms in this sum, we can observe that $[\mu_*\hat\kappa^2_r]=r(\lfloor\frac{n-1}{2}\rfloor+1)\delta_1(1)$.
   
   Case: $n\equiv 0\pmod{4}$.
   Then $n$ is even and $\binom{n}{2}$ is even.
   So $H^1(K_{12})$ is generated by $\delta_2(1)$ (which generates the $R/\binom{n}{2}R$-factor) and elements of the form $[\kappa^2_r]$, with $r\in R[2]$.
   We have that $\delta_2(1)$ maps to $2\delta_1(1)$, and $[\kappa^2_r]$ maps to $(n-1)[\kappa^1_r]=[\kappa^1_r]$, because $n-1$ is odd and $r$ has order $2$.
   Since $n$ is even, $H^1(S^1)$ is $R/nR\oplus R[2]$.  
   We see that the $R[2]$-factor maps isomorphically to the $R[2]$-factor, contributing nothing to the kernel or cokernel.
   So the kernel and cokernel of $H^1(K_{12})\to H^1(S^1)$ are the same as that of $R/\binom{n}{2}R\to R/nR$, with the map being scaling by $2$.

   Case: $n\equiv 1\pmod{4}.$
   Then $n$ is odd and $\binom{n}{2}$ is even.
   Like in the previous case, $H^1(K_{12})$ is generated by $\delta_2(1)$ and elements of the form $[\kappa^2_r]$.
   Since $n$ is odd, $H^1(S^1)$ is $R/nR$.  
   Since $r$ has order $2$, $[\kappa^2_r]$ maps to $0=(n-1)[\kappa^1_r]$.
   %Again, $\delta_2(1)$ maps to $2\delta_1(1)$.
   So the kernel and cokernel of $H^1(K_{12})\to H^1(S^1)$ are the same as that of $R[2]\oplus R/\binom{n}{2}R\to R/nR$,  by $(r,[s])\mapsto [2s]$.
   
   Case: $n\equiv 2\pmod{4}.$
   Then $n$ is even and $\binom{n}{2}$ is odd.
   In this case, $H^1(K_{12})$ is generated by $\delta_2(1)$ and elements of the form $[\kappa^2_r-\hat\kappa^2_r]$.
   $H^1(S^1)$ is $R/nR\oplus R[2]$.
   %As always, $\delta_2(1)\mapsto2\delta_1(1)$.
   In this case, $[\kappa^2_r-\hat\kappa^2_r]$ maps to $(n-1)[\kappa^1_r]+r(\lfloor\frac{n-1}{2}\rfloor+1)$, but since $n\equiv 2\pmod{4}$, we have $(\lfloor\frac{n-1}{2}\rfloor+1)=\frac{n}{2}$ is odd.
   So $[\kappa^2_r-\hat\kappa^2_r]$ maps to $[\kappa^1_r]+r\delta_1(1)$.
   So the kernel and cokernel of $H^1(K_{12})\to H^1(S^1)$ are the same as that of $R[2]\oplus R/\binom{n}{2}R\to R[2]\oplus R/nR$, by $(r,[s])\mapsto (r,[2s+r])$.
   Note that for $(r,[s])$ to be in the kernel, we must have $r=0$, and therefore $[s]$ would be in the kernel of the doubling map.
%   Since there is an $R$-module automorphism of  $R[2]\oplus R/nR$ by $(r,[s])\mapsto (r,[s+r])$, the kernel of the map on $H^1$ is the same as the kernel of the map $R[2]\oplus R/\binom{n}{2}R\to R[2]\oplus R/nR$ by $(r,[s])\mapsto(r,[2s])$.

   Case: $n\equiv 3\pmod{4}.$
   Then $n$ is odd and $\binom{n}{2}$ is odd, $H^1(K_{12})$ is generated by $\delta_2(1)$ and elements of the form $[\kappa^2_r-\hat\kappa^2_r]$, and
   $H^1(S^1)$ is $R/nR$.
   Like in the last case, $[\kappa^2_r-\hat\kappa^2_r]$ maps to $(n-1)[\kappa^1_r]+r(\lfloor\frac{n-1}{2}\rfloor+1)$, but since $n\equiv 3\pmod{4}$, we have $(\lfloor\frac{n-1}{2}\rfloor+1)=\frac{n+1}{2}$ is even.
   So $[\kappa^2_r-\hat\kappa^2_r]$ maps to $0=(n-1)[\kappa^1_r]+\frac{n+1}{2}r\delta_1(1)$.
   So the kernel and cokernel of $H^1(K_{12})\to H^1(S^1)$ are the same as that of $R[2]\oplus R/\binom{n}{2}R\to R/nR$, by $(r,[s])\mapsto [2s]$.

   Putting all this together, we see that if $n$ is even, then the quotient of $H^1(S^2)$ by the image of $H^0(R)$ is isomorphic to the kernel of the doubling map $R/\binom{n}{2}\to R/nR$, and if $n$ is odd, it is the same thing but summed with $R[2]$.
   This proves the statement about $H^1(S^2)$.

   At this point, we collect our observations about the cokernel of $\mu_*\co H^1(K_{12})\to H^1(S^1)$.
   If $n$ is congruent to $0$, $1$, or $3$ modulo $4$, then this cokernel is the same as the cokernel of $R/\binom{n}{2}R\to R/nR$ by scaling by $2$.
   So if $n$ is odd, the map is surjective and the cokernel is trivial.
   We also see that if $n\equiv 0\pmod{4}$, then the cokernel is $R/2R$.
   Now assume that $n\equiv 2\pmod{4}$.
   The cokernel is the same as the cokernel of $R[2]\oplus R/\binom{n}{2}R\to R[2]\oplus R/nR$ by $(r,[s])\mapsto (r,[r+2s])$.
   Since $n$ is even, this is the same as the cokernel of $R[2]\to R[2]\oplus R/2R$ by $r\mapsto (r,[r])$.
   One can check that the following sequence is exact:
   \[0\to R[2]\to R[2]\oplus R/2R \to R/2R\to 0,\]
   with the maps being $r\mapsto (r,[r])$ and $(r,[s])\mapsto [r+s]$.
   This means that we can recognize $\mathrm{coker}(mu_*)$ as being $R/2R$ in this case.
   So $\mathrm{coker}(\mu_*)$ is $R/2R$ if $n$ is even, and $0$ if $n$ is odd, and this is a subgroup of $H^2(S^2)$.

   Finally we compute the kernel of the map $\mu_*\co H^2(K_{12})\to H^2(S^1)$.
   Assume that $n\geq 6$.
   If $\binom{n}{2}$ is odd, then $H^2(K_{12})$ is $R[2]\oplus R/2R$, and if $\binom{n}{2}$ is even, then $H^2(K_{12})$ is $R[2]^2\oplus R/2R$.
   We have that $H^2(S^1)$ is $0$  if $n$ is odd, and $R[2]^2\oplus R/2R$ if $n$ is even.
   There is nothing more to say in the case that $n$ is odd, so we assume $n$ is even.
   We consider generators for $H^2(K_{12})$.
   
   As explained in Remark~\ref{re:K12cocycles}, for each $r\in R$, the cocycle $\alpha^2_r-\binom{n}{2}\hat\alpha^2_r$ maps to $K_{12}$, and the $R/2R$-factor of $H^2(K_{12})$ consists of cocycles of this form.
   This cocycle sends $c_i$ to $r(u-\binom{n}{2}v_{i,i+1})$ (and other cells map to $0$).
   Pushing this forward by $\mu$ yields the cocycle 
   \[c_i \mapsto r((n-1)(\sum_{j=1}^n t_j)-\binom{n}{2}(t_i+t_{i+1})).\]
   Considering this as a cocycle in $M^1$, we see that it represents $(n-1)[\alpha^1_r]$.
   Since $n$ is even and $[\alpha^1_r]\in H^2(M^1)$ has order $2$, we see that the cocycle $\alpha^2_r-\binom{n}{2}\hat\alpha^2_r$ represents a class in the kernel of $\mu_*$ if and only if $[\alpha^1_r]$ represents $0$ in $H^2(M^1)$, which is true if and only if $r\in 2R$.
   So the $R/2R$-factor of $H^2(K_{12})$ contributes nothing to the kernel.
   
   Next we consider the cocycle $\hat\beta^2_r$, for $r\in R[2]$.
   This maps $d_{ij}\mapsto r\chi(\{i,i+1,j,j+1\})$, and other cells map to $0$, where $\chi(\{i,i+1,j,j+1\})$ is the sum of the six $v_{kl}$ vectors with $k,l,\in\{i,i+1,j,j+1\}$.
   Under $\mu$, this becomes $d_{ij}\mapsto r(t_i+t_{i+1}+t_j+t_{j+1})$ (each summand appears $3$ times, but $3r=r$).
   Define a chain $f\co P_1\to M^1$ by:
   %setting $f(e_i)$ to be the sum of the $t_j$ for odd $j$, but without $t_i$ or $t_{i+1}$ (whichever is odd).
   \[f(e_i)=r(t_1+t_3+t_5+\dotsm+t_{2\lfloor\frac{n-1}{2}\rfloor+1}-t_{2\lfloor\frac{i}{2}\rfloor+1}).\]
   (So the sum here is over all odd indices, but with $t_i$ or $t_{i+1}$ omitted, whichever is odd.)
   We can check that the coboundary $\delta f=\mu_*\hat\beta^2_r$.
   If $n\equiv 2\pmod{4}$, then each $f(e_i)$ is a sum of an even number of terms, and since $r$ has order $2$, this means that $f\co P_1\to S^1$, and $[\mu_*\hat\beta^2_r]=0$ in $H^1(S^1)$.
   If $n\equiv 0\pmod{4}$, then each $f(e_i)$ is a sum of an odd number of terms, and we can recognize $\mu_*[\hat\beta^2_r]=\delta_1([\kappa^0_r])$, where $\delta_1\co H^1(R)\to H^2(S^1)$ is the connecting homomorphism from sequence~\eqref{eq:S1seq}.

   Now assume $\binom{n}{2}$ is even and consider $\beta^2_r$, for $r\in R[2]$.
   This maps $d_{ij}\mapsto ru$, and other cells to $0$.
   So $\mu$ pushes this forward to $d_{ij}\mapsto (n-1)r\sum_{j=1}^nt_j$.
   Since $n$ is even and $r$ has order $2$, this means $\mu_*[\beta^2_r]=[\beta^1_r]$, and is nontrivial.
   
   Since our classes are represented by cycles of the required order, the decomposition splits, and we get the direct sum decomposition in the statement.
   If $n=4$, the same arguments work, but it is not necessary to consider $[\beta^2_r]$.
   If $n=5$, then $H^2(S^2)\cong H^2(K_{12})$, but the $[\beta^2_r]$-classes are not distinct from the $[\hat\beta^2_r]$-classes.
   This proves the statement about $H^2(S^2)$.
   \end{proof}

\subsection{Quotients by Specht subgroups and split extensions}
In this section, we assume that $n\geq 4$.
As in the preceding section, all cohomology is of $\sg{n}$ in various coefficient modules, and we suppress $\sg{n}$ from the notation.
We recall the maps $f^i\co M^2_R\to S^i_R$, for $i\in\{0,1,2\}$, from Definition~\ref{de:fimaps}.
For each $f^i$, the image of $f^i$ is isomorphic to a quotient of $M^2_R$.
If $R$ is a quotient of $\Z$, then $f^i(M^2_R)$ is a quotient of $M^2_\Z$, and therefore a quotient of $PB_n$.
We let $f^i_R$ denote this map $PB_n\to M^2_R$.
We have an extension
\begin{equation}\label{eq:fiext}
0\to f^i(M^2_R)\to B_n/\ker(f^i_R)\to \sg{n}\to1.
\end{equation}
Similarly, we have maps $f^i\oplus f^j\co M^2_R\to S^i_R\oplus S^j_R$.
Let $f^{ij}_R$ denote the map $PB_n\to M^2_R$.
We get an extension
\begin{equation}\label{eq:fijext}
0\to (f^i\oplus f^j)(M^2_R)\to B_n/\ker(f^{ij}_R)\to \sg{n}\to1.
\end{equation}
Finally, if $m$ is a positive integer,  
we let $R=\Z/m\Z$ and we have the reduction map $M^2_\Z\to M^2_R$.
Let $\pi_m\co PB_n\to M^2_R$ be the projection.
We have an extension
\begin{equation}\label{eq:pimext}
0\to M^2_R\to B_n/\ker(\pi_m)\to \sg{n}\to1.
\end{equation}
Note that, although $\ker(\pi_2)=B_n[4]$, there is no close connection between the level subgroups and these kernels for other $m$.

We determine exactly when these extensions split.
\begin{theorem}\label{th:QbySpechtSplitting}
    Let $R$ be $\Z$ or $\Z/m\Z$ for a positive integer $m$.
    \begin{itemize}
    \item
    If $i$ is $0$, then extension~\ref{eq:fiext} splits if and only if $R=\Z/m\Z$ and $m$ is odd.
    \item 
    If $i=1$, then extension~\ref{eq:fiext} splits if $n$ is odd, or if $R=\Z/m\Z$ and $m$ is odd, but does not split otherwise.
    \item 
    If $i=2$, then extension~\ref{eq:fiext} splits if $R=\Z/m\Z$ and $m$ is odd.
%    It also splits if $n$ is even, $m\equiv 2\pmod{4}$, $R=\Z/m\Z$, and $\binom{n-1}{2}$ or $\binom{n-1}{2}-m/2$ is divisible by $\gcd(m,\binom{n-1}{2})$. 
    It does not split otherwise.
    \item
    For any distinct $i,j\in\{0,1,2\}$, extension~\ref{eq:fijext} splits if and only if $R=\Z/m\Z$ and $m$ is odd.
    \item 
    Extension~\ref{eq:pimext} splits if and only if $m$ is odd.
    \end{itemize}
\end{theorem}

We will prove the statements in this theorem separately, as lemmas.
First, we prove a proposition that handles most of the non-splitting cases.
\begin{proposition}\label{pr:oddsplitting}
    Let $m$ be an odd positive integer.
    Suppose $N\lhd B_n$ with $PB_n'\leq N\leq PB_n$, and the image of $N$ under $\pi\co PB_n\to M^2_\Z$ contains $mM^2_\Z$ (in other words, $N$ contains all the $m$th powers of whole twists $a_{ij}$).
    Then
    \[0\to \frac{PB_n}{N}\to \frac{B_n}{N}\to \sg{n}\to 1\]
    is a split extension.
\end{proposition}

\begin{proof}
    Consider the quotient map $\rho\co M^2_\Z\to PB_n/N$.
    By Corollary~\ref{co:splittingobstruction}, the extension splits if and only if $\rho_*[\phi]=0$.
    By Proposition~\ref{pr:Nakamuracocycle}, $[\phi]=[\hat\alpha^2_1]\in H^2(M^2_\Z)$, and by Proposition~\ref{pr:M2cohomology}, we know that $[\hat\alpha^2_1]$ has order $2$ (it generates a factor of $\Z/2\Z$).

    Now let $R=\Z/m\Z$.
    Notice that $\rho$ factors through the reduction map $M^2_\Z\to M^2_R$.
    But Proposition~\ref{pr:M2cohomology} implies that $H^2(M^2_R)=0$, since $m$ is odd.
    So $\rho_*[\phi]=0\in H^2(PB_n/N)$, and therefore the extension splits.
\end{proof}

\begin{lemma}
    Extension~\eqref{eq:pimext} splits if and only if $m$ is odd.
\end{lemma}

\begin{proof}
    If $m$ is odd, then Proposition~\ref{pr:oddsplitting} applies and the extension splits.

    So assume $m$ is even.
    Let $\rho\co M^2_\Z\to M^2_R$ by reduction modulo $m$.
    Again by Corollary~\ref{co:splittingobstruction}, it is enough to show that $\rho_*[\phi]$ is nontrivial in $H^2(M^2_R)$. 
    But $\rho_*[\phi]$ is $[\hat\alpha^2_1]\in H^2(M^2_R)$.
    Since $[\hat\alpha^2_1]$ generates a submodule of $H^2(M^2_R)$ isomorphic to $R/2R$, we have that $\rho_*[\phi]$ is nontrivial.
\end{proof}

\begin{lemma}
    If $i=0$, then extension~\eqref{eq:fiext} splits if and only if $R=\Z/m\Z$ and $m$ is odd.
\end{lemma}

\begin{proof}
    The proof is similar to the previous lemma.  
    We assume that $R=\Z$ or $R=\Z/m\Z$ with $m$ even, since otherwise Proposition~\ref{pr:oddsplitting} applies.
    By evaluating on basis elements from $P_2$, we recognize that $f^0\circ \hat\alpha^2_1=\alpha^0_1$.
    So extensions~\ref{eq:fiext} splits if and only if $[\alpha^0_1]=0$.
    But $R$ is $\Z$ or $\Z/m\Z$ with $m$ even, so $[\alpha^0_1]$ generates a copy of $\Z/2\Z$ in $H^2(f^0(M^2))$ by Proposition~\ref{pr:untwistedcohomreps}.
    So it is nontrivial, and the extension does not split.
\end{proof}

\begin{lemma}
    Let $i=1$.
    Extension~\eqref{eq:fiext} splits if $n$ is odd, or if $R=\Z/m\Z$ and $m$ is odd, but does not split otherwise.
\end{lemma}

\begin{proof}
    This is similar to the previous two lemmas, but is more involved.
    Again, we start by assuming that $R=\Z$ or $R=\Z/m\Z$ with $m$ even.
    By Proposition~\ref{pr:S1cohomology}, we see that $H^2(S^1)$ is trivial if $n$ is odd.
    So extension~\ref{eq:fiext} splits in this case.

    So we assume that $n$ is even.
    We find $f^1_*\hat\alpha^2_1$ has $d_{ij}\mapsto 0$, $b_i\mapsto 0$, and 
    \[c_i\mapsto n(t_i+t_{i+1})-\sum_{k=1}^nt_k.\]
    To show that this cocycle represents a nontrivial class in $H^2(f^1(M^2))$, we consider it as a cocycle for $H^2(M^1)$.
    In $M^1$, we can recognize this as $-\alpha^1_1$, plus a coboundary.
    Specifically, the cochain $P_1\to M^1$ by $e_i\mapsto \frac{n}{2}(t_i+t_{i+1})$, has coboundary $c_i\mapsto n(t_i+t_{i+1})$, $d_{ij}\mapsto 0$, $b_i\mapsto 0$.
    So in $H^2(M^1)$, we have $f^1_*[\hat\alpha^2_1]=-[\alpha^1_1]$.
    This is nontrivial since $[\alpha^1_1]$ generates a copy of $\Z/2\Z$ in $H^2(M^1)$, by Proposition~\ref{pr:M1cohomology}, since $R$ is $\Z$ or $\Z/m\Z$ with $m$ even.
    Since $f^1_*[\hat\alpha^2_1]$ maps to a nontrivial class in $H^2(M^1)$, it is nontrivial in $H^2(f^1(M^2))$, and the extension does not split.
\end{proof}

\begin{lemma}
    Suppose $n$ is odd and $R=\Z$ or $R=\Z/m\Z$ for $m$ even.
    Then for $i=2$, extension~\eqref{eq:fiext} does not split.
\end{lemma}

\begin{proof}
    Since $n$ is odd, $b=1/2$ in our formula for $f^2$.
    We consider $f^2\circ \hat\alpha^2_1$, which has $d_{ij}\mapsto 0$, $b_i\mapsto 0$, and 
    \[c_i\mapsto \binom{n-1}{2}(v_{i,i+1})-\frac{n-1}{2}(w_i+w_{i+1})+u.\]
    (It would be twice this if $n$ were even.)
    Our goal is to show that this represents a nontrivial class in $H^2(f^2(M^2))$.
    First we consider $f^2\circ \hat\alpha^2_1$ as a cocycle in $M^2$.
    The cochain $P_1\to M^2$ defined by $e_i\mapsto w_i$ has coboundary given by $d_{ij}\mapsto 0$, $b_i\mapsto 0$, and $c_i\mapsto w_i+w_{i+1}$.
    So in $H^2(M^2)$, 
    \[f^2_*[\hat\alpha^2_1]=\binom{n-1}{2}[\hat\alpha^2_1]+[\alpha^2_1].\]
    By Proposition~\ref{pr:M2cohomology}, the classes $[\hat\alpha^2_1]$ and $[\alpha^2_1]$ generate a copy of $(R/2R)^2$ in $H^2(M^2)$.
    By our hypothesis on $R$, we have that $R/2R\cong \Z/2\Z$.
    So this is a nontrivial class in $H^2(M^2)$.
    Then $f^2_*[\hat\alpha^2_1]$ must also be nontrivial in $H^2(f^2(M^2))$, and we are done.
\end{proof}

\begin{lemma}\label{le:coboundf2hatalpha}
    If $n$ is even, then $f^2_*\hat\alpha^2_1$ is the coboundary of the cochain $\zeta\co P^1\to K_{12}$ defined by 
    \[\zeta\co e_i\mapsto \binom{n-1}{2}v_{i,i+1}-(n-1)w_i+u.\]
\end{lemma}

\begin{proof}
    Since $n$ is even, $b=1$ in our formula for $f^2$.
    Then $f^2\circ \hat\alpha^2_1$ has $d_{ij}\mapsto 0$, $b_i\mapsto 0$, and 
    \[c_i\mapsto (n-1)(n-2)(v_{i,i+1})-(n-1)(w_i+w_{i+1})+2u.\]
    To verify that $\delta\zeta=f^2\circ\hat\alpha^2_1$, we use the formulas for the boundaries of $c_i$, $d_{ij}$, and $b_i$ from Definition~\ref{de:Snres}.
    It is almost immediate that $\zeta(\partial d_{ij})=0$, because if $j>i+1$, then $s_i$ fixes $\zeta(e_j)$ and $s_j$ fixes $\zeta(e_i)$.
    The computation for $b_i$ is more subtle; it follows from
    \[
    \begin{split}
    (1-s_i&+s_{i+1}s_i)v_{i+1,i+2}-(1-s_{i+1}+s_is_{i+1})v_{i,i+1}\\
    &=v_{i+1,i+2}-v_{i,i+2}+v_{i,i+1}-v_{i,i+1}+v_{i,i+2}-v_{i+1,i+2}=0,
    \end{split}
    \]
    and
    \[(1-s_i+s_{i+1}s_i)w_{i+1}-(1-s_{i+1}+s_is_{i+1})w_{i}
    =w_{i+1}-w_{i}+w_{i}-w_i+w_i-w_{i+1}=0,
    \]
    together with the fact that all $s_i$ fix $u$.
    Finally, we check that
    \[\zeta(\partial e_i)=(s_i+1)\zeta(e_i)=2\binom{n-1}{2}v_{i,i+1}-(n-1)(w_i+w_{i+1})+2u,\]
    since $s_i$ fixes both $v_{i,i+1}$ and $u$.
    Since this is the formula for $f^2\circ \hat\alpha^2_1(e_i)$, this proves that it is the coboundary of $\zeta$.
    
    The lemma also asserts that $\zeta$ takes values in $K_{12}$.
    To verify this, we compose $\zeta$ with the map $M^2\to R$ by $v_{ij}\mapsto 1$.
    It sends $\zeta(e_i)$ to
    \[\binom{n-1}{2}-(n-1)^2+\binom{n}{2}
    =(n-1)(\frac{n-2}{2}-(n-1)+\frac{n}{2})=0.
    \]
    Since $K_{12}$ is defined as the kernel of the this map $M^2\to \R$, our codomain for $\zeta$ is correct.
\end{proof}

We need another submodule of $M^2$, in order to keep track of the difficulties of building cochains that satisfy the congruences from Proposition~\ref{pr:f2M2congruences}.
\begin{definition}
    Let $M^2_\equiv$ be the submodule of $M^2$ consisting of vectors that satisfy the systems of congruences~\eqref{eq:coeffcongruence} and~\eqref{eq:polytabloidcongruence} in Proposition~\ref{pr:f2M2congruences}.
\end{definition}

\begin{lemma}\label{le:M2equivgens}
    $M^2_\equiv$ is generated by vectors of the following three kinds:
    the vector $u$, the vectors $b(n-1)w_i$ for all $i$, and the vectors $b(n-1)(n-2)v_{ij}$ for all choices of $i$ and $j$.
\end{lemma}

\begin{proof}
    We argue by double inclusion between $M^2_\equiv$ and the submodule generated by these elements.
    First note that all three of the given kinds of vectors are actually in $M^2_\equiv$.
    For the other direction, suppose $v=\sum_{ij}c_{ij}v_{ij}$ is in $M^2_\equiv$.
    By subtracting $c_{n-1,n}u$, we may assume that $c_{n-1,n}=0$.
    Then for all $i,j$, we have $b(n-1)|c_{ij}$.
    In particular, for each $i$ from $1$ to $n-2$, we have that $c_{i,n}w_i$ is a multiple of $b(n-1)w_i$.
    So by subtracting these elements from $v$, we may assume $c_{i,n}=0$ for all $i$.
    Likewise, we know that $c_{1,n-1}w_n$ is a multiple of $b(n-1)w_n$.
    By adding $c_{1,n-1}(w_n-u)$, we may assume that $c_{1,n-1}=0$, without undoing our work to make $c_{1,n}=\dotsm=c_{n-1,n}=0$.
    For any $i$ with $1<i<n$, 
    \[c_{1,i}+c_{n-1,n}\equiv c_{i,n}+c_{1,n-1}\pmod{n-2},\]
    meaning that $b(n-1)(n-2)$ divides $c_{1,i}$.
    Now, for any $i,j$ with $1\leq i<j\leq n-1$ and $(i,j)\neq (1,n)$, we have
    \[c_{i,j}+c_{1,n}\equiv c_{1,i}+c_{j,n}\pmod{n-2},\]
    which implies that $b(n-1)(n-2)$ divides $c_{i,j}$.
    Since $i,j$ were arbitrary, this means that $v$ can be expressed as a sum of elements of the form $b(n-1)(n-2)v_{ij}$.
\end{proof}

\begin{lemma}\label{le:congruenceobstruction}
    Suppose $n$ is even and $R$ satisfies the following two conditions:
    \begin{itemize}
        \item $\binom{n-1}{2}$ has order $2$ in $R/(2\binom{n-1}{2}R)$, and
        \item nothing in $R[2]$ maps to the same element as $\binom{n-1}{2}$ in $R/(2\binom{n-1}{2}R)$.
    \end{itemize}
    Then for $i=2$, extension~\eqref{eq:fiext} does not split.
\end{lemma}

\begin{remark}
    The conditions are satisfied if $R=\Z$.
    If $R=\Z/m\Z$, the first condition says that $2$ divides $m$ to at least the same power that it divides $2\binom{n-1}{2}$.
    The second condition says that $\binom{n-1}{2}$ is not congruent to $m$ or $m/2$ modulo $\gcd(m,2\binom{n-1}{2})$.
    These conditions are independent of each other.
\end{remark}

\begin{proof}[Proof of Lemma~\ref{le:congruenceobstruction}]
    %This involves some subtleties about trying to force a cochain cobounding $f^2_*\hat\alpha^2_1$ to satisfy the congruences in Proposition~\ref{pr:f2M2congruences}.
    By Proposition~\ref{pr:f2M2congruences},  $f^2(M^2)\subseteq M^2_\equiv.$
    Let $M^2_\sim$ denote the quotient $M^2/M^2_\equiv$.
    Let $I$ be the ideal $(n-1)(n-2)R$.
    Let $L$ denote $M^2_\equiv/IM^2$.
    Since $IM^2\leq f^2(M^2)$, we have the following commuting diagram with exact rows
    \[
    \xymatrix{
    0 \ar[r] & M^2_\equiv \ar[r] \ar[d] & M^2  \ar[r] \ar[d] & M^2_\sim \ar[r] \ar@{=}[d] & 0 \\
    0 \ar[r] & L \ar[r] & M^2_{R/I} \ar[r] & M^2_\sim \ar[r] & 0.
    }    
    \]

    We want to show that $f^2_*[\hat\alpha^2_1]\in H^2(M^2_\equiv)$ is nontrivial.
    Let $\pi\co M^2\to M^2_\sim$ be the quotient.
    We know that $f^2_*[\hat\alpha^2_1]$ maps to the trivial class in $M^2$, so it is the image under the connecting homomorphism of $[\pi_*\zeta]\in H^1(M^2_\sim)$.
    We want to show that $[\pi_*\zeta]$ is nontrivial, so we will show that it is the image of a nontrivial class from $H^1(M^2_{R/I})$ that is not in the kernel of the map induced by $M^2_{R/I}\to M^2_\sim$.
    Let $\rho\co M^2_{R/I}\to M^2_\sim$ be the quotient.
    
    We claim that $\rho_*[\hat\kappa^2_{\binom{n-1}{2}}]=[\pi_*\zeta]$.
    This class is well defined since $\binom{n-1}{2}$ has order $2$ in $R/I$ (by our hypothesis on $R$).
    We know that $\hat\kappa^2_{\binom{n-1}{2}}$ sends $e_i$ to $\binom{n-1}{2}v_{i,i+1}$, so 
    \[\pi_*\zeta(e_i)-\rho_*\hat\kappa^2_{\binom{n-1}{2}}(e_i)=u-(n-1)w_i.\]
    But $u$ and $(n-1)w_i$ are both in $M^2_\equiv$, so these elements are equal in $M^2_\sim$.
    Further, $[\hat\kappa^2_{\binom{n-1}{2}}]$ is nontrivial, since by Proposition~\ref{pr:M2cohomology}, it sits inside a submodule of the form $(R/I)[2]$ in $H^2(M^2_{R/I})$, and our hypothesis on $R$ ensures that $\binom{n-1}{2}$ is nontrivial in $R/I$.

    Now we show that $[\hat\kappa^2_{\binom{n-1}{2}}]$ is not in the image of $H^1(L)\to H^1(M^2_R/I)$.
    We analyze the structure of $L$.
    By Lemma~\ref{le:M2equivgens}, since $n$ is even, we know that $M^2_\equiv$ is generated by elements of the following three forms: $u$, $(n-1)w_i$, and $(n-1)(n-2)v_{ij}$.
    Since $L$ is the image of $M^2_\equiv$ after taking the quotient by $IM^2$, we have eliminated the elements of the form $(n-1)(n-2)v_{ij}$.
    So there is an injective map $R/I\to L$ by  $1\mapsto u$.
    We can recognize the cokernel of this map.
    The remaining elements of $L$ are linear combinations of elements $(n-1)w_i$, and we recognize a map from $M^1$ to $L/Ru$.
    Let $\Delta\subseteq (n-1)M^1_{R/(n-2)R}$ denote the submodule spanned by $\sum_{i=1}^nt_i$; we have an exact sequence
    \[0\to R/I\to L \to \frac{(n-1)M^1_{R/(n-2)R}}{\Delta}\to 0.\]
    However, by Proposition~\ref{pr:M1cohomology}, we know that $H^0$, $H^1$, and $H^2$ of $(n-1)M^1_{R/(n-2)R}$ are given by representatives that take values in $\Delta$.
    So the long exact sequence for
    \[0\to \Delta\to (n-1)M^1_{R/(n-2)R} \to \frac{(n-1)M^1_{R/(n-2)R}}{\Delta}\to 0\]
    tells us that $H^0$ and $H^1$ of $\frac{(n-1)M^1_{R/(n-2)R}}{\Delta}$ are trivial.
    Then the long exact sequence for the preceding short exact sequence tells us that the map $R/I\to L$ induces an isomorphism on $H^1$.
    So $H^1(L)$ consists of classes of the form $[\kappa^2_r]$, for $r\in (R/I)[2]$.
    In particular, no such class can map to $[\hat\kappa^2_{\binom{n-1}{2}}]$ (by Proposition~\ref{pr:M2cohomology}, they are in distinct summands of $H^2(M^2_{R/I})$).
    This implies that $[\pi_*\zeta]\in H^1(M^2_\sim)$ is nontrivial.
    
    To finish the argument, we need to show that $[\pi_*\zeta]$ is not in the image of $H^1(M^2)\to H^1(M^2_\sim)$.
    By Proposition~\ref{pr:M2cohomology}, we know that $H^2(M^2)$ consists of classes of the form $[\kappa^2_r+\hat\kappa^2_s]$ for $r,s\in R[2]$.
    For such a class to map to $[\pi_*\zeta]$, we would need to have $r=0$ and $s=\binom{n-1}{2}$ in $R/I$.
    But no element congruent to $\binom{n-1}{2}$ modulo $I$ can have order two in $R$, by our hypothesis on $R$.
    So $[f^2_*\hat\alpha^2_1]$ is the image of a nontrivial class that is not in the kernel of the connecting homomorphism.
    So it is nontrivial in $H^2(M^2_\equiv)$, and therefore it is also nontrivial in $H^2(f^2(M^2))$.
\end{proof}

\begin{lemma}
    Suppose $n$ is even, $m$ is divisible by $4$, $R=\Z/m\Z$, and $R$ fails the hypotheses of Lemma~\ref{le:congruenceobstruction}.
    Then for $i=2$, extension~\eqref{eq:fiext} does not split.
\end{lemma}

\begin{proof}
    Our terms are the same as in the previous lemma.
    We want to show that, under these hypotheses, we can produce a cochain $\zeta'$ in $M^2_\equiv$ that cobounds $f^2_*\hat\alpha^2_1$.
    First suppose that $\binom{n-1}{2}=0$ in $R/I$.
    By the definition of $\zeta$ in Lemma~\ref{le:coboundf2hatalpha}, all the summands are in $M^2_\equiv$: $(n-1)w_i$ and $u$ are in $M^2_\equiv$ regardless of this hypothesis, and $\binom{n-1}{2}v_{i,i+1}$ is in $M^2_\equiv$ since $\binom{n-1}{2}\in I$. 
    So in this case, we take $\zeta'=\zeta$.

    Now suppose that there is $r\in R[2]$ with $r=\binom{n-1}{2}$ in $R/I$.
    If $r=0$, then $\binom{n-1}{2}\in I$, and the previous case applies.
    So suppose $r=m/2$.
    Define $\zeta'=\zeta-\hat\kappa^2_r\co P_1\to M^2$.
    Then $\zeta'$ still cobounds $f^2_*\hat\alpha^2_1$, but 
    \[\zeta'(e_i)=(\binom{n-1}{2}-r)v_{i,i+1}-(n-1)w_i+u.\]
    Since $\binom{n-1}{2}-r\in I$, this is in $M^2_\equiv$.

    So in either case, we have defined $\zeta'\co P_1\to M^2_\equiv$  with $\delta\zeta'=f^2_*\hat\alpha^2_1$.
    We want to show that it is impossible to modify $\zeta'$ to satisfy the last condition in Proposition~\ref{pr:f2M2congruences}; this amounts to showing that $f^2_*\hat\alpha^2_1$ maps to a nontrivial class in a submodule of $M^2_\equiv$.

    Define a function $\sigma\co M^2_\equiv\to \Z/2\Z$ as follows: for $v=\sum_{ij}c_{ij}v_{ij}$ in $M^2_\equiv$, pick any three distinct $i,j,k$ from $\{1,2,\dotsc,n\}$, and define $\sigma(v)=c_{ij}+c_{ik}+c_{jk}$.
    We claim that $\sigma$ is a well defined $\Z\sg{n}$-module homomorphism.
    (Really $\sigma$ maps to $R/2R$, but since $m$ is even, this is $\Z/2\Z$.)
    To see that it is well defined, let $i,j,k,l$ be distinct indices.
    Since $v\in M^2_\equiv$, we have $c_{ij}+c_{kl}\equiv c_{il}+c_{kj}\pmod{n-2}$.
    Since $n$ is even, these are also congruent modulo $2$.
    So $c_{ij}+c_{jk}\equiv c_{il}+c_{kl}\pmod{2}$, and therefore $c_{ij}+c_{jk}+c_{ik}=c_{il}+c_{kl}+c_{ik}$ in $\Z/2\Z$.
    In particular, $\sigma(v)$ is the same whether we compute it using $i,j,k$ or $i,l,k$.
    Repeated substitutions show us that all triples give the same value.
    So $\sigma$ is well defined, and we might as well have defined it as $\sigma(v)=c_{12}+c_{13}+c_{23}$.
    This makes it clear that it is a homomorphism of abelian groups.
    Since it is the same using any triple, this implies that it is $\sg{n}$-equivariant.

    Further, we know that $\sigma$ is surjective since $\sigma(u)=1$.
    So we have a short exact sequence
    \[0\to \ker(\sigma)\to M^2_\equiv \to \Z/2\Z\to 0.\]
    By Proposition~\ref{pr:f2M2congruences}, $f^2(M^2)\subseteq \ker(\sigma)$.
    So $f^2_*[\hat\alpha^2_1]\in H^2(\ker(\sigma))$ is the image of $[\sigma_*\zeta']\in H^1(\Z/2\Z)$ under the connecting homomorphism.
    In evaluating $\sigma(\zeta'(e_i))$, we can always choose $j$ and $k$ distinct from $i$ and $i+1$, and evaluate $\sigma$ using the coefficients on $v_{ij}$, $v_{ik}$, and $v_{jk}$.
    From the formula for $\zeta'$, these are always $2-n$, $2-n$, and $1$, which tells us that $\sigma(\zeta'(e_i))=1$ for all $i$.
    So $\sigma_*\zeta'=\kappa^0_1$, which is nontrivial.
    
    We need to show that $[\kappa^0_1]$ is not in the image of $\sigma_*\co H^1(M^2_\equiv)\to H^1(\Z/2\Z)$.
    Suppose that $\gamma\co P^1\to M^2_\equiv$ is a cocycle.
    Then $\gamma$ is also a cocycle $P^1\to M^2$, and by Proposition~\ref{pr:M2cohomology}, it is $\gamma=\kappa^1_r+\hat\kappa^1_s+\delta v$, for some $r,s\in R[2]$ and $v\in M^2$.
    Then
    \[\gamma(e_i)=ru+sv_{i,i+1}+(s_i-1)v.\]
    Notice that $(s_i-1)v$ contributes nothing to $\sigma(\gamma(e_i))$, since $\sigma$ is equivariant and $\Z/2\Z$ has the trivial action.
    If we compute $\sigma$ using $i$ and $i+1$, then we get $\sigma(\gamma(e_i))=3r+s$, but if we compute $\sigma$ using $i$, $j$, and $k$ with $j,k\notin\{i,i+1\}$, then we get $\sigma(\gamma(e_i))=3r$.
    Since $\sigma$ is well defined, we deduce that $s$ maps to $0$ in $\Z/2\Z$.
    So for $\sigma_*[\gamma]$ to equal $[\kappa^0_1]$, we must have $r\in R[2]$ with $3r=1$ in $\Z/2\Z$.
    The only nontrivial element of $R[2]$ is $m/2$, but our hypothesis that $m$ is divisible by $4$ implies that $m/2$ is even, and therefore that $3m/2$ is $0$ in $\Z/2\Z$.
    So $[\kappa^0_1]$ is not in the image of $\sigma_*$, and therefore $[f^2_*p\hat\alpha^2_1]\in H^2(\ker(\sigma))$ is the image of a nontrivial class that is not in the kernel of the connecting homomorphism.
    So $f^2_*[\hat\alpha^2_1]$ is nontrivial in both $H^2(\ker(\sigma))$ and in $H^2(f^2(M^2))$.
\end{proof}

\begin{lemma}
    Suppose $n$ is even, $m\equiv 2\pmod{4}$, $R=\Z/m\Z$, and $\binom{n-1}{2}$ or $\binom{n-1}{2}-m/2$ is divisible by $\gcd(m,\binom{n-1}{2})$. 
    Then for $i=2$, extension~\eqref{eq:fiext} does not split.
\end{lemma}

\begin{proof}
    With notation as in the previous lemma, let $K_{\equiv}=K_{12}\cap \ker(\sigma)$.
    We have an exact sequence
    \[0\to K_\equiv \to \ker(\sigma) \to (n-1)R\to 0.\]
    The second map is the restriction of $v_{ij}\mapsto 1$.
    The image of this map is correct since $(n-1)w_1$ and $2u$ are both in $\ker(\sigma)$, and $(n-1)w_1\mapsto (n-1)^2$, and $2u\mapsto 2\binom{n}{2}$, and the greatest common divisor of these is $n-1$.
    We know $f^2(M^2)\subseteq K_\equiv$, so we aim to show that $f^2_*[\hat\alpha^2_1]$ is nontrivial in $H^2(K_\equiv)$.

    First we show that $f^2_*[\hat\alpha^2_1]$  is the image of a nontrivial class in $H^1((n-1)R)$.
    We break into cases.
    If $n\equiv 0\pmod{4}$, then $\binom{n}{2}$ is even and $\binom{n-1}{2}$ is odd.
    We know that $\binom{n-1}{2}=0\in R/I$ if and only if it is a multiple of $\gcd(m,2\binom{n-1}{2})$, but this is impossible, since it is odd.
    So we are in the case that $\binom{n-1}{2}-m/2$ is in $I$.
    Set $\zeta'$ equal to $\zeta+\hat\kappa^2_{m/2}+\kappa^2_{m/2}$.
    Then $\partial\zeta'=f^2_*\hat\alpha^2_1$, and $\zeta'$ maps to $\ker(\sigma)$.
    But since $\binom{n}{2}$ is even, $[\zeta']$ maps to $[\kappa^0_{m/2}]$ in $H^1((n-1)R)$.

    Next suppose $n\equiv 2\pmod{4}$.
    Then $\binom{n}{2}$ is odd and $\binom{n-1}{2}$ is even.
    Then $\binom{n-1}{2}-m/2$ is odd, so we must have $\binom{n-1}{2}\in I$.
    Set $\zeta'$ equal to $\zeta+\kappa^2_{m/2}$.
    Then $\partial\zeta'=f^2_*\hat\alpha^2_1$, and $\zeta'$ maps to $\ker(\sigma)$.
    But since $\binom{n}{2}$ is odd, $[\zeta']$ maps to $[\kappa^0_{m/2}]$ in $H^1((n-1)R)$.

    To finish the argument, we argue like in the last lemma.
    If $\gamma\co P_1\to \ker(\sigma)$ is a cocycle, then it is also a cocycle for $M^2$, and is $\kappa^2_r+\hat \kappa^2_s+\delta v$, with $r,s\in R[2]$ and $v\in M^2$.
    However, since $\gamma$ maps to $M^2_\equiv$, we must have $s=0$, since the other option is $s=m/2$, and the polytabloid $v_{12}+v_{34}-v_{14}-v_{23}$ would have a nontrivial inner product with $\gamma(e_1)$, modulo $n-2$.
    But for $\gamma$ to map to $\ker(\sigma)$, we must also have $r=0$, since otherwise $\sigma(\gamma(e_1))=3r=r\neq 0$.
    So $\gamma$ is a coboundary, and $[\gamma]$ maps to $0$ under $H^1(\ker(\sigma))\mapsto H^1((n-1)R)$.
    
    This shows that $f^2_*[\hat\alpha^2_1]$ is nontrivial in $H^2(K_\equiv)$ and therefore also in $H^2(f^2(M^2))$, so that the extension does not split.
\end{proof}

\begin{lemma}
    For any distinct $i,j\in\{0,1,2\}$, extension~\eqref{eq:fijext} splits if and only if $R=\Z/m\Z$ and $m$ is odd.
\end{lemma}

\begin{proof}
    Again we start by assuming that $R=\Z$ or $R=\Z/m\Z$ with $m$ even, since otherwise the extension splits by Proposition~\ref{pr:oddsplitting}.
    Without loss of generality, assume that $i\neq 1$.
    By the other parts of Theorem~\ref{th:QbySpechtSplitting}, we have that $f^i_*[\hat\alpha^2_1]$ is nontrivial in $H^2(f^i(M^2_R))$.
    But we have the first coordinate projection $(f^i\oplus f^j)(M^2_R)\to f^i(M^2_R)$, and $f^i\co M^2_\R\to f^i(M^2_R)$ factors through this.
    So if $(f^i\oplus f^j)_*[\hat\alpha^2_1]$ is trivial, then so is $f^i_*[\hat\alpha^2_1]$, a contradiction.
    So extension~\eqref{eq:fijext} does not split.
\end{proof}

\section{Further questions}
In this section, we assume $n\geq 4$.
First of all, are there other ways of getting at the Specht subgroups?
\begin{question}
    Are there natural ways to define the Specht subgroups in terms of topology of surfaces, or in terms of elementary considerations about the group theory of $B_n$?
    Are there connections between the Specht subgroups and matrix representations of braid groups?
\end{question}

Further, we know very little about the basic properties of the Specht subgroups.
\begin{question}
    Which of $N_1$, $N_2$, $N_{01}$, and $N_{02}$ are finitely generated?
    Are any of the Specht subgroups finitely presentable?
\end{question}
We know from Proposition~\ref{pr:N12fg} that $N_{12}$ is finitely generated, and we know from Proposition~\ref{pr:N0nfg} that $N_0$ is not finitely generated, but this is all we know.

\begin{question}
    What are the abelianizations of the Specht subgroups?
\end{question}
Again, we know from Proposition~\ref{pr:N0nfg} that $N_0$ has an infinite-rank abelianization, but the picture is less clear for the other Specht subgroups.
Of course, there are many other group-theoretic questions we can ask, but we think these questions are a good place to start.

One question, which gives context for the structure of extensions of the form~\eqref{eq:topic}, is the following.
\begin{question}
    What is the structure of $H^2(\sg{n};PB_n/N)$, where $N$ is a Specht subgroup?
\end{question}
We have partial progress on this question, but for brevity we have left it out of this paper.
There are also things to explore about the cohomology of $\sg{n}$ in other subquotients of $PB_n$, and there is much to explore in the total range of subgroups $N\lhd B_n$ with $PB'_n\leq N\leq PB_n$.

We end with a general question that we have barely started to consider.
\begin{question}
    What can we say about the set of extensions
    \[1\to PB_n/N\to B_n/N\to \sg{n}\to 1,\]
    where $N\lhd B_n$ with $N\leq PB_n$, but with $PB'_n$ not assumed to be contained in $N$, so that $PB_n/N$ is not necessarily abelian?
\end{question}

\bibliography{qbgesgbib}
\bibliographystyle{acm}

\vspace{0.2cm}
\small
 \noindent \textbf{Matthew B. Day,} {\sc Department of Mathematical Sciences, 309 SCEN, University of Arkansas, Fayetteville, AR 72701, U.S.A.} \\
\noindent  e-mail: {\tt matthewd@uark.edu}
\\
\\
 \noindent \textbf{Trevor Nakamura}, {\sc formerly of the Department of Mathematical Sciences, University of Arkansas}
\\
\noindent  e-mail: {\tt trevornakamura1994@gmail.com}

\end{document}